\documentclass{amsart}
\usepackage[utf8]{inputenc}
\usepackage{amssymb, tipa}
\usepackage{graphicx}
\usepackage{xcolor}
\usepackage{tikz}
\usepackage{csquotes}
\usepackage[english]{babel}
\usepackage[fixlanguage, datename]{babelbib}
\usepackage{hyperref}

\newtheorem{theorem}{Theorem}[section]
\newtheorem{maintheorem}{Theorem}
\newtheorem{lemma}[theorem]{Lemma}
\newtheorem{remark}[theorem]{Remark}
\newtheorem{corollary}[theorem]{Corollary}
\newtheorem{definition}[theorem]{Definition}
\newtheorem{example}[theorem]{Example}

\newtheorem{fact}[theorem]{Fact}

\newcommand{\set}[1]{\{#1 \}}

\DeclareMathOperator{\tp}{tp}
\DeclareMathOperator{\qftp}{qftp}

\title{Model theory of homogeneous D-sets}
\author{Felipe Estrada and John Goodrick}
\thanks{The second author was supported by the grant INV-2023-162-2840 from the Faculty of Sciences of the Universidad de los Andes while carrying out this research}

\begin{document}

\begin{abstract}
    We explore several model-theoretic aspects of $D$-sets, which were studied in detail by Adeleke and Neumann. We characterize ultrahomogeneity in the class of colored $D$-sets and classify unbounded order-indiscernible sequences in such structures. We use these results to provide a characterization of distal colored $D$-sets and prove that all colored $D$-sets with quantifier elimination are dp-minimal.
\end{abstract}

\maketitle

\section{Introduction}

A $D$-set is a relational structure which consists of a set and a quaternary relation $D$, axiomatized so that the set may be interpreted as the leaves of a tree and the relation $D(x,y,z,w)$ (henceforth written as $D(xy;zw)$ or as the shorthand $xy|zw$) may be interpreted as the the statement that the unique path from $x$ to $y$ is disjoint from the unique path from $z$ to $w$. These structures have been extensively developed by Adeleke and Neumann \cite{adeleke1998relations}, along with related concepts such as $C$-sets and $B$-sets. These tree-like structures arise naturally in the study of infinite primitive Jordan permutation groups (see \cite{adeleke1996classification}, \cite{adeleke1996primitive}). In this paper we will study the model-theoretic properties of general $D$-sets, which are known to not be stable. In addition to pure $D$-sets, we shall also study $D$-sets augmented with a partition into unary predicates, which we will refer to as \emph{colored} $D$-sets, with a particular eye towards characterizing ultrahomogeneity, distality and dp-minimality in colored $D$-sets. Braunfeld, Macpherson and Almazaydeh have also studied the model theory of $D$-sets, with a focus on homogeneity and orbital growth rate \cite{almazaydeh2024omega}.

The first question we will consider is the classification of ultrahomogeneous colored $D$-sets. Recall:

\begin{definition}
    A first-order structure $\mathcal{M}$ is called \emph{ultrahomogeneous} if every partial isomorphism between finitely-generated substructures of $\mathcal{M}$ can be extended to an automorphism of $\mathcal{C}$.
\end{definition}

Our first main result is to give a classification of ultrahomogeneous colored $D$-sets.

\begin{maintheorem}
    \label{ultrahomogeneity theorem}
        A countable proper colored $D$-set $(\Omega; P_1, \ldots, P_n)$ is ultrahomogeneous if, and only if, $\Omega$ is dense and regular and for each infinite sector $\Sigma$ of every splitting and for every color $i$, $\Sigma \cap P_i \neq \emptyset$.
\end{maintheorem}

All of the terms in the statement of Theorem~\ref{ultrahomogeneity theorem} will be defined precisely in Section 2, but to give the reader an idea: imagine, for the sake of intuition, that our $D$-set arises as the set of leaves of a tree, with $D(xy;zw)$ asserting the disjointness of paths. Then each internal node of this tree has ``branches,'' and the set of leaves on any such branch is a \emph{sector}. To be \emph{proper} means that no sector consists of a single leaf, to be \emph{dense} means that there is always an internal node between any two non-intersecting paths, and to be \emph{regular} means that every internal node has the same number of branches.

Our other main results concern indiscernible sequences and dp-minimality of colored $D$-sets. In NIP (or dependent) theories, Shelah isolated a cardinal-valued rank for partial types, analogous to weight in stable theories, known as \emph{dp-rank} (\cite{shelahstrong}). Theories in which all $1$-types have dp-rank $1$ are \emph{dp-minimal}; this class includes all strongly minimal and all o-minimal theories, all theories of colored linear orders, and $p$-adically closed fields. For more on dp-minimality and its importance, see \cite{simon2015guide}.

We add another class of treelike structures to the growing list of examples of dp-minimal theories:

\begin{maintheorem}
    \label{dpmin}
        Let $T$ be the theory of a $D$-set augmented with unary predicates such that $T$ has quantifier elimination. Then $T$ is dp-minimal.

        In particular, every countable proper homogeneous colored $D$-set is dp-minimal.
\end{maintheorem}

Following the present introduction, in Section 2 we will recall the formal axiomatization of $D$-sets and we will establish a couple of their basic properties. In Section 3 we will extensively develop the correspondence between $D$-sets and (graph-theoretic) trees, which provides an intuitive grounding and helpful tools for the rest of the study. In Section 4, we will use the previously-developed tools to fully characterize ultrahomogeneous $D$-sets. We shall also introduce a weaker notion of homogeneity called $P$-homogeneity and we shall characterize $D$-sets with this property.

To finish, in Section 5 we will characterize indiscernible sequences which are not bounded above or below in colored $D$-sets and use this partial characterization to give sufficient and necessary conditions for a colored $D$-set with quantifier elimination to be distal, and further prove that such colored $D$-sets are always dp-minimal.

\section{Preliminaries of D-sets and Splittings}

We will consider $D$-sets in the sense of Adeleke and Neumann (see \cite{adeleke1998relations}). Here, $D(w,x; y,z)$ is a $4$-ary relation with the following definition taken from their book.

\begin{definition}\label{Definition of D-set}
    A \emph{$D$-set} is a set $\Omega$ endowed with a $4$-ary relation $D(wx;yz)$ and satisfying the following properties for all $w, x, y,$ and $z$ in $\Omega$:

    \begin{align*}
        \textup{(D1)} & \hspace{0.2in} D(wx;yz) \rightarrow \left(D(xw; yz) \wedge D(yz; wx) \right) \\
        \textup{(D2)} & \hspace{0.2in} D(wx;yz) \rightarrow \neg D(wy;xz)\\
        \textup{(D3)} & \hspace{0.2in}  D(wx;yz) \rightarrow \forall v\left(D(vx;yz) \vee D(wx;yv) \right)\\
        \textup{(D4)} & \hspace{0.2in} \left(w \neq y \wedge x \neq y \right) \rightarrow D(wx;yy).
    \end{align*}

    It is a \emph{proper} $D$-set if it contains at least three elements and in addition

    \begin{align*}
        \textup{(D5)} & \hspace{0.2in} \left(|\{w, x, y\}| = 3\right) \rightarrow \exists z \left(z \neq y \wedge D(wx;yz)\right).
    \end{align*}

    A $D$-set is \emph{dense} if it has at least two elements and

    \begin{align*}
        \textup{(D6)} & \hspace{0.2in} D(wx;yz) \rightarrow \exists v \left(D(vx;yz) \wedge D(wv;yz) \wedge D(wx;vz) \wedge D(wx;yv)\right).
    \end{align*}
\end{definition}

\begin{example}
\label{flowers}
For each nonzero cardinal $\kappa$, an \emph{improper} $D$-set $F_\kappa$ (a ``flower'' of size $\kappa$) can be defined as follows: as a set, $F_\kappa$ consists of $\kappa$ elements, and for $w, x, y, z \in F_\kappa$, the relations $D(wx; yz)$ holds if and only if either $w = x$ and $w \notin \{y,z\}$, or $y = z$ and $y \notin \{w,x\}$. This corresponds to the $D$-relation on the leaves of a tree with a single internal note and $\kappa$ leaves. It is routine to check the axioms (D1)-(D4) all hold of $F_\kappa$. Note also that any flower $F_\kappa$ is ultrahomogeneous, since its $D$-relation is definable in the language with only equality, so that $\textup{Aut}(F_\kappa)$ is simply the set of all permutations of the underlying set.
\end{example}

The following result from Adeleke and Neumann (\cite{adeleke1998relations}, \S 22, (D9)) is occasionally helpful:

\begin{fact}[Transitivity Lemma]
    If $D(ab; xy)$ and $D(ab;yz)$ then $D(ab, xz)$.
\end{fact}

In order to study $D$-sets, we introduce a key concept which generalizes the concept of structural partitions from the book of Adeleke and Neumann.

\begin{definition}\label{splitting}
    If $\Omega$ is any D-set, a \emph{splitting of $\Omega$} is a partition $ \set{\Sigma_i: i \in \alpha}$ of $\Omega$ where $\alpha$ is some (possibly infinite) ordinal, $\alpha \geq 2$, satisfying the following conditions:
    \begin{enumerate}
    \item If $a, b$ are any two elements of some $\Sigma_i$ and $c,d$ are any two  elements of $\Omega \setminus \Sigma_i$, then $D(ab; cd)$; and
    \item If $a, b, c,$ and $d$ each belong to different sets of the partition, then $\neg D(ab; cd)$.
    \end{enumerate}

    A \emph{sector} is an element of a splitting.
\end{definition}

For intuition, if $(\Omega, D)$ is the $D$-set of leaves of a tree, then each internal node $x$ gives rise to splitting: for each edge connected to $a$, the set of all leaves which connect to $a$ through that edge form a sector, and every edge gives rise to a splitting of size 2 composed of the sets of leaves that connect to the edge through each of its vertices.

Occasionally, it becomes useful to distinguish splittings into two sectors from splittings into three or more (which are the same as Adeleke and Neumann's structural partitions). Hence the following definition:

\begin{definition}\label{Types of splitting}\phantom{ }
    \begin{enumerate}
        \item A \emph{node splitting} is a splitting into more than 2 sectors.
        \item An \emph{edge splitting} is a splitting into two sectors.
        \item A \emph{true edge splitting} is an edge splitting which cannot be refined to a node splitting, and such that no sector is a singleton 
    \end{enumerate}  
\end{definition}

The motivation for this nomenclature comes from the natural interpretation of $D$-relations as relations on the leaves of a tree. A node splitting is a partition of the leaves induced by the connected components one obtains after removing an inner node from the tree, and an edge splitting is the same but after removing an edge instead

The following notion of regularity for $D$-sets will be needed for our characterization of ultrahomogeneous colored $D$-sets in Section 4.

\begin{definition}
        A $D$-set is \emph{regular} if all node splittings have the same cardinality. A $D$-set is $d$-regular if all node sectors have size $d$.
\end{definition}

Here are some useful basic facts about sectors from Adeleke and Neumann, adapted to our terminology:

\begin{fact}\phantom{ }\label{basic facts}
    \begin{enumerate}
        \item (Lemma 24.1) If $\Omega$ is a proper $D$-set, then every sector of a node splitting is infinite.
        \item (Theorem 24.2) If $x, y,$ and $z$ are distinct elements of a $D$-set, then there is a node splitting in which they are in different sectors.
        \item (Corollary 25.2 (1)) No two distinct node splittings may share a sector.
        \item (Corollary 25.2 (2)) If $\Sigma_1$ and $\Sigma_2$ are any two sectors of a $D$-set $\Omega$, possibly from different splittings, then at least one of the following holds: (i) $\Sigma_1 \subseteq \Sigma_2$, (ii) $\Sigma_2 \subseteq \Sigma_1$, (iii) $\Sigma_1 \cap \Sigma_2 = \emptyset$, or (iv) $\Sigma_1 \cup \Sigma_2 = \Omega$.
    \end{enumerate}
\end{fact}

\begin{lemma}\label{Sector disjoint from set}
Let $\Omega$ be a $D$-set with $|\Omega| \geq 3$, $A$ be a finite subset of $\Omega$ and $c \in \Omega \setminus A$. Then there is a sector $\Sigma$ of some splitting of $\Omega$ such that $c \in \Sigma$ and $A \cap \Sigma = \emptyset$.
\end{lemma}
\begin{proof}
Let $A = \set{a_1, \dots, a_m}$. We construct the sector $\Sigma$ recursively by constructing a family of sectors which contain $c$. Let $\Sigma_0$ be any sector which contains $c$, which we can guarantee the existence of by applying Fact \ref{basic facts}(2) to $c$ and two other elements of $\Omega$ and extracting the sector containing $c$ from the resulting splitting.

Given a sector $\Sigma_n$ containing $c$ such that $\Sigma_n \cap \set{a_1, \dots, a_n}  = \emptyset$, then if $a_{n+1} \not\in \Sigma_n$ we take $\Sigma_{n+1} = \Sigma_n$, otherwise let $b$ be an arbitrary point in $\Omega \setminus \Sigma_n$. Again by Fact \ref{basic facts}(2) there exists a splitting such that $a_{n+1}, b$ and $c$ are in different sectors. Let $\Sigma_{n+1}$ be the one which contains $c$.

Notice that $\Sigma_n$ and $\Sigma_{n+1}$ have a nonempty intersection (since both contain $a_{n+1}$), their union is not all of $\Omega$ (since neither contains $b$), and $\Sigma_{n+1} \not\subseteq \Sigma_n$ (since $a_{n+1} \in \Sigma_n \setminus \Sigma_{n+1})$. This eliminates three out of four possibilities from Fact 5.5(4) and so the remaining one must be true: $\Sigma_{n+1} \subseteq \Sigma_n$ and so $c \in \Sigma_{n+1} \cap \set{a_1, \dots, a_{n+1}} = \emptyset$. Thus $\Sigma = \Sigma_m$ fulfills the statement.
\end{proof}

    \begin{lemma}
        If $\mathcal{C}$, $\mathcal{D}$ are splittings of a $D$-set $\Omega$ such that $\mathcal{C}$ is a strict refinement of $\mathcal{D}$, then $\mathcal{D}$ is an edge splitting, $\mathcal{C}$ is a node splitting, and $\mathcal{C}$ and $\mathcal{D}$ have at least one sector in common.
    \end{lemma}
    \begin{proof}
        First let us prove they have at least one sector in common. If they did not, then in particular there would be two sectors $D_1, D_1 \in \mathcal{D}$ and four sectors $C_1, C_1', C_2, C_2' \in \mathcal{C}$ such that $C_i, C_i' \subseteq D_i$ for $i = 1,2$, but then if $a \in C_1,b \in C_1', c \in C_2, d \in C_2'$ then $ab|cd$ since $a,b$ are in the same $\mathcal{D}$-sector and $c,d$ are not, but that relation cannot hold since $a,b,c,d$ are all in different $\mathcal{C}$-sectors.

        Now since they have a sector in common, if both were node splittings they would be equal by Fact \ref{basic facts}(3), so one must be an edge splitting, and clearly it must be $\mathcal{D}$.
    \end{proof}

The following Lemma is exceptionally useful:

   \begin{lemma}[The One Sector Lemma]\label{the one sector lemma}
        Let $\mathcal{C}$ and $\mathcal{D}$ be two different splittings of a $D$-set $\Omega$. There is a sector $D_0$ of $\mathcal{D}$ such that all but one sectors of $\mathcal{C}$ are contained in $D_0$.
    \end{lemma}
    \begin{proof}
        If $\mathcal{C}$ is a refinement of $\mathcal{D}$, then by the previous lemma $\mathcal{D}$ is an edge splitting consisting of one sector of $\mathcal{C}$ and the union of all other sectors of $\mathcal{C}$, and so the result holds. If $\mathcal{D}$ is a refinement of $\mathcal{C}$ then $\mathcal{C}$ is an edge splitting and the result is trivial.

        Now suppose neither refines the other. Then $\mathcal{C}$ and $\mathcal{D}$ share no sector in common (by Fact~\ref{basic facts} (3) if both are node splittings, and this is trivial if one is an edge splitting). Enumerate the (possibly infinitely many) sectors of $\mathcal{C}$ as $C_1, C_2, \cdots$ and take an element $a_i$ from each $C_i$. For each $i$ let $D_i$ be the sector of $\mathcal{D}$ which contains $a_i$. We shall prove there is a $D_0$ such that $D_i = D_0$ for all except at most one $i$, and this $D_0$ will fulfill the conditions of the lemma.

        If all $D_i$'s are in fact equal to the same $D_0 \in \mathcal{D}$, take any element in $\Omega \setminus D_0$, which lies in some $C_i$, and replace $a_i$ with this element so that there is at least one $D_i$ different from the rest.
        
        Now, since not all $D_i$'s are the same, suppose without restriction that $D_1$ and $D_2$ are different. First we shall prove that $C_1 \subseteq D_1$ or $C_2 \subseteq D_2$. Suppose neither holds and take $b_1 \in C_1 \setminus D_1$ and $b_2 \in C_2 \setminus D_2$. Then $a_1b_1|a_2b_2$ since $a_1, b_1 \in C_1$ and $a_2,b_2 \not \in C_1$, but this relation is impossible since $a_1$ and $b_1$ are in different $\mathcal{D}$-sectors and $a_2, b_2$ are as well. We conclude that $C_1 \subseteq D_1$ or $C_2 \subseteq D_2$ (suppose the latter). By an analogous argument either $D_1 \subseteq C_1$ or $D_2 \subseteq C_2$. Since the splittings share no sectors we must conclude the former (otherwise $D_2 = C_2$), and these must be strict containments, so then we may assume $D_1 \subsetneq C_1$ and $C_2 \subsetneq D_2$. Choose some $b \in C_1 \setminus D_1$.

        Let $i > 2$ and let $c \in C_i$. Since $a_1, b \in C_1$ but neither $a_2$ nor $c$ belongs to $C_1$, we have that $a_1 b|c a_2$. But since $b \notin D_1$, the elements $a_1$ and $b$ are in different $\mathcal{D}$-sectors, so for this relation to hold it must be that $c$ and $a_2$ are in the same $\mathcal{D}$-sector, that is to say $c \in D_2$. This proves $C_i \subseteq D_2$ for all $i > 2$, and we already have $C_2 \subseteq D_2$, so we may take $D_0 = D_2$ and the lemma holds.
    \end{proof}

    \begin{remark}
        The One Sector Lemma implies that for any pair of distinct splittings $\mathcal{C}_1$, $\mathcal{C}_2$, there exist unique sectors $\Sigma_1 \in \mathcal{C}_1, \Sigma_2 \in \mathcal{C}_2$ with the property that $\Sigma_1 \cup \Sigma_2 = \Omega$, and furthermore $\Sigma_1$ is precisely the sector of $\mathcal{C}_2$ which contains all but one sectors of $\Sigma_2$, with the exception being $\Sigma_2$, and vice-versa.
    \end{remark}

\section{D-sets and trees}

There is a known relationship between $D$-sets and trees, one direction of which is clear by the following fact, which is folklore:

\begin{fact}\label{Tree leaves are D-sets}
    On any tree, the quaternary relation $D(a b ; c d)$ on the leaves of the tree, which holds if the path from $a$ to $b$ is disjoint from the path from $c$ to $d$, is a $D$-relation
\end{fact}

Our goal for this section will be to prove a converse for Fact \ref{Tree leaves are D-sets} for finite $D$-sets. That is to say, we shall prove that every finite $D$-set may be obtained from a finite tree in the manner so described. The tools developed along the way will be useful for characterizing ultrahomogeneity.

\begin{definition}
    Given a $D$-set $\Omega$ and three distinct points $a,b,c$, the \emph{branch of $a$ over $\{b,c\}$}, denoted by $B(a; bc)$, is the set of $x \in \Omega \setminus \set{a}$ such that $bc|ax$.
\end{definition}

\begin{lemma}\label{shrinkage lemma}
    If $d \in B(a,bc)$, then $B(a,dc) \subsetneq B(a,bc)$.
\end{lemma}
\begin{proof}
    \begin{equation*}
        \begin{split}
            x \in B(a,dc) 
            & \Longrightarrow dc|ax\\ 
            &\stackrel{(D3)}{\Longrightarrow} bc|ax \vee dc|ab \quad \textrm{(latter false since $d \in B(a,bc) \Rightarrow bc|ad$ instead)}\\
            &\Longrightarrow bc|ax\\
            &\Longrightarrow x \in B(a,bc).
        \end{split}
    \end{equation*}
    Thus $B(a,dc) \subseteq B(a,bc)$ and the two sets cannot be equal since $d \in B(a,bc) \setminus B(a,dc)$.
\end{proof}

\begin{lemma}\label{induced splitting}
Suppose that $A \cup \{e\}$ is a finite $D$-set (with $e \notin A$, $|A| \geq 2$). Then there is a unique splitting $\mathcal{C}$ of $A$ satisfying the following conditions:
\begin{enumerate}
    \item If $a, b$ are any two elements of some $\Sigma \in \mathcal{C}$ and $c$ is some element of $A \setminus \Sigma$, then $D(ab; ce)$; and
    \item If $a$ and $b$ belong to different $\Sigma$'s, then $\neg D(ab; xe)$ for any $x \in A$.
\end{enumerate}
\end{lemma}
\begin{proof}
    Define an equivalence relation $\sim$ on $A$ by $a \sim b$ if and only if there exists $x \in A$ such that $ab|ex$. Reflexivity is clear from the definition. Symmetry follows from (D1). To see that it is transitive, let $a \sim b$ and $b \sim c$. We shall consider the (up to symmetry) three $D$-relations which may hold on $\set{a,b,c,e}$, plus the case of there being no relations, separately:
\begin{itemize}
    \item Case 1, $ab|ce$:

    Let $x$ be a witness to $b\sim c$, that is, $x$ is such that $bc|ex$. Since $ab|ce$, by (D3) either $xb|ce$ or $ab|ex$. However, the first one cannot be since $bc|ex$ holds instead, hence $ab|ex$, and now by the Transitivity Lemma $ab|cx$ and so $a \sim c$

    \item Case 2, $ac|be$:

    Let $x$ be a witness to $b \sim c$, that is $bc|ex$. As before, by (D3) either $cx|be$ or $ac|ex$. The first contradicts $bc|ex$, so $ac|ex$ and so $a \sim c$.

    \item Case 3, $ae|bc$:

    Identical to the first case, switching $a$ with $c$.

    \item Case 4, no $D$-relations hold on $\set{a,b,c,e}$:

    Let $x$ be a witness to $b \sim c$, that is $bc|ex$. Then by (D3) either $ac|ex$ or $bc|ea$. The latter contradicts the hypothesis, hence $ac|ex$ and so $a \sim c$.
    
\end{itemize}

    We shall define the sets $\Sigma_i$ as the equivalence classes of this relation. For this to have a chance of being a splitting at all, the relation must have at least two classes, so let us check that $\sim$ is not the universal relation. Since $A$ is finite, by iteratively applying Lemma \ref{shrinkage lemma} to $B(e, a_0b)$ for arbitrary $a_0, b \in A$ we obtain that there exists an $a \in A$ such that $B(e, ab)$ is empty, but this is equivalent to stating that $a \not\sim b$ and so $\sim$ is not universal.
    
    Now we officially define $\mathcal{C}$ as the equivalence classes of $\sim$. We shall prove that this partition is a splitting and satisfies the lemma. Let $a,b \in \Sigma \in \mathcal{C}$ and $c,d \not\in \Sigma$, then there is an $x$ such that $ab|xe$, and so by (D3) either $ac|ex$ or $ab|ec$. The first implies $a \sim c$ so $ab|ec$. This already gives us condition (1) of this lemma. From here, by (D3) once again we have either $ad|ec$ or $ab|cd$, and the first implies $a \sim d$, so $ab|cd$ and so this partition satisfies condition (1) of Definition \ref{splitting}.

    Now suppose that $D(ab; cd)$ holds for some tuple of distinct elements. By applying (D3) we obtain either $eb|cd$ or $ab|ce$, which implies either $a \sim b$ or $c \sim d$. Therefore no quadruple of elements in distinct $\Sigma$'s can be $D$-related and so condition (2) of Definition \ref{splitting} is satisfied. Additionally, if $D(a, b; x, e)$ holds for any $x \in A$ then by definition $a \sim b$, so if $a \not\sim b$ then $\lnot D(a, b,c,e)$ and so condition (2) of this lemma is satisfied as well. We conclude that the classes of $\sim$ are indeed a splitting which satisfies the lemma.

    To see the uniqueness of this splitting, let $\mathcal{D}$ be a splitting which satisfies the lemma. Let $a,b \in D \in \mathcal{D}$ and $c \in D' \in \mathcal{D}$ with $D \neq D'$. Then by hypothesis $ab|ec$, but then $a \sim b$. On the other hand if $a \sim b$ then $ab|xe$ for some $x$, but then $a$ and $b$ cannot be in different $D$'s or it would contradict (2). We conclude that any two elements are related by $\sim$ if and only if they lie on the same $D$. That is to say $\mathcal{D}$ is the same partition as $\mathcal{C}$.
\end{proof}

The splitting defined in the Lemma above is called the \emph{splitting of $A$ induced by $e$}.

\begin{remark}
    Suppose two elements $e$ and $e'$ induce the same splitting on a finite $D$-set $A$, then $A \cup \set{e}$ is isomorphic to $A \cup \set{e'}$ with the isomorphism $id_A \cup \set{(e, e')}$.
\end{remark}
\begin{remark}
    If $e$ induces a splitting $\mathcal{C}$ on a set $A$, then $\mathcal{C} \cup \set{\set{e}}$ is a splitting of $A \cup \set{e}$.
\end{remark}

\begin{lemma}\label{BAF lemma 2}
    Let $A$ be a finite D-set and $\mathcal{C}$ be a splitting of $A$. Then for any sector $\Sigma \in \mathcal{C}$ and every $a \in \Sigma$  there exists a (not necessarily unique) $b \in \Sigma$ such that if $c \in \Sigma$ and $d \not\in \Sigma$, then $\lnot D(ab; cd)$.
\end{lemma}
\begin{proof}
    First, take an arbitrary pair $b \in \Sigma$, $d \not\in \Sigma$. Notice that if $d'$ is any element of $\Omega \setminus \Sigma$, then $B(d', ab) = B(d, ab)$, since $ab|dd'$ and so if $ab|d'x$ then $ab|dx$ by transitivity. 
    
    Now, by successive application of Lemma \ref{shrinkage lemma} there exists a choice of $b_0 \in \Sigma$ such that $B(d,ab_0)$ does not contain any elements of $\Sigma$, since if there is $b' \in B(d,ab) \cap \Sigma$ then $B(d,ab')$ contains strictly fewer elements of $\Sigma$.

    We conclude that all $B(d, ab_0)$ for $d \not \in \Sigma$ are disjoint from $\Sigma$, which is equivalent to $b = b_0$ satisfying the conclusion of the lemma.
\end{proof}

 We call a $b \in \Sigma$ satisfying the conditions of this lemma for an $a \in \Sigma$ a \emph{complementary element to $a$} (with respect to the splitting).

\begin{lemma}
    Let $A \cup \set{e}$ be a finite $D$-set ($e \not \in A$, $|A| \geq 2)$, $\mathcal{C}$ be the splitting of $A$ induced by $e$, $\Sigma \in \mathcal{C}$ be a sector, $a \in \Sigma$ be arbitrary and $b \in \Sigma$ be complementary to $a$. Then for any $x \in \Sigma$, $\lnot D(ab;xe)$.

\end{lemma}
\begin{proof}
    Suppose $ab|xe$ for some $x \in \Sigma$ and let $y \in A\setminus \Sigma$ be arbitrary, then by (D3) either $yb|xe$ or $ab|xy$. The former contradicts $xb|ye$ which holds since $e$ induces $\mathcal{C}$, while the latter contradicts that $a$ and $b$ are complementary.
\end{proof}

\begin{lemma}[Sector Indifference Principle]\label{sector indifference principle}
    Let $A \subseteq \Omega$ be a set, $\mathcal{C}$ be a splitting and $\mathcal{C}(A)$ denote the union of the sectors of $\mathcal{C}$ which contain elements of $A$. Then all elements of $\Omega \setminus \mathcal{C}(A)$ have the same quantifier-free type over $A$.
\end{lemma}
\begin{proof}
    By symmetry, it suffices to show for any three $a_1, a_2, a_3 \in A$ and any two $b, b' \in \Omega \setminus \mathcal{C}(A)$, that $D(ba_1;a_2a_3) \Leftrightarrow D(b'a_1;a_2a_3)$.

    Suppose $ba_1|a_2a_3$. Then by (D3) either $b'a_2|a_3a_4$ or $b a_1|a_2b'$. We claim the latter to be impossible, since by hypothesis, in the set $\set{a_1, a_2, b, b'}$ the only two pairs of elements which could share a sector of $\mathcal{C}$ are $\set{a_1, a_2}$ and $\set{b, b'}$. Either pair sharing a sector yields the relation $a_1a_2|bb'$ and neither pair sharing a sector implies no $D$-relations hold on $\set{a_1, a_2, b, b'}$. Both possibilities are incompatible with $b a_1|a_2b'$, so we conclude that $b'a_2|a_3a_4$. That is to say $D(ba_1;a_2a_3) \Rightarrow D(b'a_1;a_2a_3)$ and the other direction follows symmetrically.
\end{proof}

\begin{theorem}\label{D set extension lemma}
    Let $A $ be a finite $D$-set ($|A| \geq 2$), $\mathcal{C}$ be a splitting of $A$. Then $A$ may be extended to a $D$-set $A \cup \{e\} $ such that $\mathcal{C}$ is the splitting of $A$ induced by $e$. 

    Furthermore, the $D$-relation in the extended $D$-set is uniquely characterized by the following properties, where $a,b,c \in A$:
    \begin{enumerate}    
        \item $D(ab; ce)$ never holds if $a$ and $b$ lie in different components of $\mathcal{C}$.

        \item Whenever $a$ and $b$ are in the same component $\Sigma$ of $\mathcal{C}$, $D(a,b, c, e)$ holds if and only if  $D(ab;cx)$ holds for some (equivalently, any) $x \in A \setminus \Sigma$. (Note that if $c \not \in \Sigma$, then $D(ab;ce)$ holds automatically).
    \end{enumerate}
\end{theorem}
\begin{proof}
    We must define the truth value of the $D$-relations which involve the new element $e$. if $e$ appears more than once in a quadruple, then the value the $D$-relation is set so as to fulfill (D3) (i.e. it is true if of the form $ee|xy$ or $xy|ee$ for $x,y \neq e$, false otherwise). If $e$ appears exactly once, note that we need only define the truth value if $e$ is in the fourth entry, since the values when $e$ is in another entry may be deduced from that by the requirement that (D1) is fulfilled. With that in mind, the characterization in the statement is enough to completely determine the $D$-relation.

    To avoid confusion later on in the proof, we shall refer to the two cases in the statement as regimes. Notice that by the Sector Indifference Principle, the truth value of $D(ab;cx)$ in regime (2) is independent of our choice of $x$. Hence the ``equivalently, any'' parenthetical.

    We check the $D$-set axioms. (D1) is trivial: we only defined the truth value when $e$ is in the fourth entry and its truth value for other cases is left to be decided so as to satisfy (D1), so there is nothing to check. (D4) is even more trivial by how we defined the case of an $e$ in multiple entries. (D2) is relatively straightforward; it suffices the check that $D(ab;ce)$ and $D(ac;be)$ never hold simultaneously and all other cases follow by applications of (D1). If it were the case that both held, then both would have to be in regime (2), so that $a,b,c$ are in the same $\Sigma$, but then $D(ab;cx)$ and $D(ac;bx)$ hold for some (in fact, any) $x \in A \setminus \Sigma$, which contradicts (D2) for $D$ on $A$.

    The last axiom to check is (D3), which in simplified notation states
    $$ab|cd \rightarrow ( xb|cd \vee ab|cx).$$ This proof must split into  cases to account for the possible variations of the axiom. The axiom contains 5 universally-quantified variables, and in principle we must check for cases where we replace any combination of them with $e$. Fortunately any cases in which two or more variables are replaced with $e$ (or more generally, in which two variables share a value) are trivial, and two of the cases are symmetric to another one, so in the end we really need only check three. Hereafter $a,b,c,d,x$ are assumed to be distinct elements of $A$. We shall use the notation $a \sim b$ to denote that $a$ and $b$ are in the same sector of $\mathcal{C}$, as in Lemma \ref{induced splitting}.
    
    \begin{itemize}
        \item Case 1: Check that $ab|cd \rightarrow (eb|cd \vee ab|ce)$. That is to say (D3) holds when $e$ replaces $x$.

        If $ab|cd$, then notice that it is not possible that the $a,b,c,d$ are all in distinct $\Sigma$'s, since in such a case no $D$-relation would hold between them. Also notice that either $a \sim b$ or $c \sim d$, since otherwise the only relations that could hold are from  $a \sim c, a\sim d, b \sim c$ or $b \sim d$. We shall only show the first is impossible (the others are analogous): If $a \sim c$ while $a \not\sim b$ and $c \not\sim d$, then we'd have $ac|bd$ (since $a,c$ are in the same $\Sigma$ and $b,d$ are not) in contradiction with the hypothesis.

        In summary, we are forced to have at least one $\sim$ relation in $\set{a,b,c,d}$, and it must include $a\sim b$ or $c \sim d$. We may suppose without loss it is $a \sim b$ since the other case is analogous by switching the roles of $a$ with $d$ and of $b$ with $c$.

        Let $y \in A$ be such that $y \not \sim a$, then by applying (D3) to $(a,b,c,d,y)$ we obtain that $yb|cd$ or $ab|cy$. If the latter holds we are done since then $ab|ce$ by the regime (2), so suppose that it does not. It must be that $c \sim a$, as otherwise $ab|cy$ automatically, and also $d \sim a$, as otherwise we could take $y = d$ and obtain that $D(a,b,c,d)$ is true and false. So then $a \sim b \sim c \sim d \not \sim y$ and so
        $$eb|cd \stackrel{\text{(D1)}}{\Longleftrightarrow} cd|be \stackrel{\text{def}}{\Longleftrightarrow}cd|by \stackrel{\text{(D1)}}{\Longleftrightarrow} yb|cd $$
        and this last one is true, so $eb|cd$ and we are done.

        \item Case 2: Check that $ab|ce \rightarrow (xb|ce \vee ab|cx)$ (Or equivalently by symmetry: $eb|cd \rightarrow (xb|cd \vee eb|cx)$). That is to say (D3) holds when $e$ replaces $d$ (or $a$).

        We must be in regime (2) for $ab|ce$ to hold, that is to say $a\sim b$ and there is a $y \not \sim a$ such that $ab|cy$, but then by (D3) applied to $(a, b, c, y, x)$ we have that $xb|cy$ or $ab|cx$. If the latter we are done, so suppose that $xb|cy$. If $x \sim b$ then $xb|ce$ by regime (2) since $xb|cy$ and $x\sim b \not \sim y$, and if $x \not\sim b$ then $ab|cx$ since $ab|ce$ and $a\sim b \not \sim x$

        \item Case 3: Check that $ab|ed \rightarrow (xb|ed \vee ab|ex)$ (Or equivalently by symmetry: $ae|cd \rightarrow (xe|cd \vee ae|cx)$. That is to say (D3) holds when $e$ replaces $c$ (or $b$).

        Notice that by (D1) and regime (2) we have that $ab|ed$ if and only if $a\sim b$ and $ab|yd$ for some (any) $y \not\sim a$. Like before we apply (D3) to this and obtain that $xb|yd$ or $ab|yx$. If the latter we are done since this is equivalent to $ab|ex$, so suppose that $xb|yd$. As before, if $x \sim b$ then $xb|ed$ by regime (2) and if $x \not \sim b$ then $ab|ex$ automatically.
    \end{itemize}

    Thus, axiom (D3) holds in the extension, so that $A \cup \set{e}$ is indeed a $D$-set. It now remains only to check that $e$ induces the given splitting, but this evident, since condition (1) of Lemma \ref{induced splitting} is guaranteed by the parenthetical in the second regime and condition (2) is guaranteed by regime (1).

\end{proof}

Now we are ready to prove the main theorem of this section, which serves as the motivation for the terminology introduced in Definition \ref{Types of splitting}.

\begin{theorem}\label{D-sets are trees}
    Any finite $D$-set $\Omega$ is isomorphic to the set of leaves of a finite tree with no binary nodes with the relation in Theorem \ref{Tree leaves are D-sets}. Furthermore, this tree is unique up to isomorphism and satisfies the following:
    \begin{enumerate}
        \item There exists a one-to-one correspondence between the inner nodes of $T$ and the splittings of $\Omega$ with three or more sectors. This correspondence maps nodes of degree $d$ to splittings with $d$ sectors, and is given by mapping an inner node $\mu$ to the restriction to the leaves of $T$ of the set of connected components of $T \setminus \{\mu\}$.

        \item There exists a one-to-one correspondence between the edges of $T$ and the splittings of $\Omega$ with two sectors, given by mapping an edge $\epsilon$ to the restriction to the leaves of $T$ of the set of connected components of $T \setminus \{\epsilon\}$.
    \end{enumerate}
\end{theorem}
\begin{proof}
    For existence, we proceed by induction on the size of $\Omega$. Conditions (1) and (2) will follow immediately from the construction. The cases of $|\Omega| = 0, 1, 2$ are all trivial, since $D$-sets of this size are isomorphic to the leaves of, respectively, the empty tree, the singleton and the tree with two nodes.

    Now we write our $D$-set as $A \cup \set{e}$ with $e \not \in A$, $|A| \geq 2$. By induction hypothesis $(A, D)$ is isomorphic to the leaves of some finite tree $T$. We identify the elements of $A$ with the leaves of $T$. It remains only to ``attach'' the remaining leaf $e$ to the tree at either some edge or some node. Let $\set{\Sigma_i: i = 1, \dots, d}$ be the splitting of $A$ induced by $e$. For each $\Sigma_i$ define $\overline{\Sigma_i}$ to be the convex hull of $\Sigma_i$ in $T$. 
    
    Notice that the $\overline{\Sigma_i}$'s are pairwise disjoint, since if there is an $x \in \overline{\Sigma_i} \cap \overline{\Sigma_j}$ then $x$ is in the path between two points $a,b \in \Sigma_i$ and also between two points $c,d \in \Sigma_j$, but then $D(ab;cd)$ does not hold which contradicts the partition's construction
    
    We bifurcate on the value of $d$:
    \begin{itemize}
        \item $d = 2$.

        Then the splitting contains two sectors $\Sigma_1$ and $\Sigma_2$. Notice that every inner node of the tree must be contained in either $\overline{\Sigma_1}$ or $\overline{\Sigma_2}$. To see this, notice that every inner node $\mu$ has degree at least three, so if you choose one random leaf node from each connected component of $T \setminus \set{\mu}$, then by the pigeonhole principle there must be some pair of such leaf nodes in the same $\Sigma_i$, so that $\mu \in \overline{\Sigma_i}$.

        Now, since $\overline{\Sigma_1} \cup \overline{\Sigma_2}$ contains every vertex, it must contain all but one of the edges, since if it contained every edge, there would be one edge with each node on a different $\overline{\Sigma_i}$, which is absurd, and if there were two or more such edges there would be more than one connected component (since a node between two such edges would have to be in a different connected component to nodes on the other sides of each edge). 
        
        Create a new binary inner node on this distinguished edge, and attach to it a new node, which we identify with $e$. We shall check that this new tree is isomorphic to $\Omega$ together with the next case after we define it.

        \item $d > 2$

        We claim that every inner node of the tree save exactly one is contained in some $\overline{\Sigma_i}$. If there was no such node, then by the previous argument there would be a single distinguished edge not in any $\Sigma_i$, but then there would only be two connected components. If there were two such nodes, say $\mu$ and $\nu$, Then $T \setminus \set{\mu}$ has at least three connected components, and so it has at least two components which do not contain $\nu$. Pick two random leaf nodes $a,b$ from distinct such components, and analogously pick $c,d$ from distinct connected components of $T \setminus \set{\nu}$ which do not contain $\mu$. Then all of $a,b,c,d$ must come from different $\Sigma_i's$, but the path from from $a$ to $b$ is disjoint from the path from $c$ to $d$ (they are separated by the path from $\mu$ to $\nu$), so that $ab|cd$ contradicting the fact that $\mathcal{C}$ is a splitting.

        Then the union of all $\overline{\Sigma_i}$'s contains the whole tree save one node $x$ and its connecting edges. Increase the degree of this distinguished node by attaching one new node, which we identify with $e$.
    \end{itemize}

    The only thing left to check is that the new tree, augmented with the new node $e$, is indeed isomorphic to the $D$-set $A \cup \set{e}$. This will follow from Theorem \ref{D set extension lemma} by proving that the new $D$-relation fulfills the characterization given in that lemma. 
    
    To see this, let $a,b,c\in A$ and let $\mu$ be $e$'s unique neighbor node. If $a$ and $b$ lie in different sectors of the partition, then the path between them has a segment outside any $\overline{\Sigma_i}$ (since it must cross from one segment to another at some point, and convex hulls of sets of nodes cannot be adjacent), that is to say the path must pass through $\mu$, as does any path from $e$ to any other node, so that $D(ab;ce)$ never holds as required by the lemma. If, on the other hand, $a,b$ are both in some $\Sigma_i$, then the path between them lies within $\overline{\Sigma_i}$ and so cannot pass through $\mu$. If $c \not \in \Sigma_i$ then the path from $c$ to $x$ is not in $\overline{\Sigma_i}$ and so $D(ab;ce)$ does not hold. If $c \in \Sigma_i$ then let $x  \in A \setminus \Sigma_i$; the difference between the path from $c$ to $e$ and the path from $c$ to $x$ is only that the segment from $\mu$ to $e$ changes to a segment from $\mu$ to $x$, and these segments lie outside $\Sigma_i$, so this cannot change whether either of these paths intersect with the path from $a$ to $b$. That is to say $D(ab;ce)$ holds if and only if $D(ab;cx)$ does as per the lemma.

    Thus, by Theorem \ref{D set extension lemma}, the leaves of this new tree must coincide with a $D$-set extension of $A$ by one point $e$ such that $e$ induces $\mathcal{C}$ on $A$, which is precisely what $\Omega$ is, so indeed $\Omega$ is isomorphic to the leaves of this tree, and by induction we are done with existence.

    To see the uniqueness and the one-to-one correspondences, suppose $T$ and $T'$ are two trees without binary edges which both generate $\Omega$ through their leaves. Identify $\Omega$ with the leaves of $T$ and notice that for every internal node $\mu$ of $T$, the partition of $\Omega$ obtained by restricting the connected components of $T \setminus \set{\mu}$ to the leaves, is in fact a splitting of $\Omega$ with at least 3 elements. Furthermore, by the previous argument any such splitting also determines a node with this property, thus there is a natural bijection between splittings with three or more elements and inner nodes of $T$. Similarly, there is a natural bijection between edges of $T$ and splittings of $\Omega$ with two sectors. Repeating this for $T'$ we may obtain a natural bijection between the inner nodes of $T$ and the inner nodes of $T'$ given by mapping the inner node or edge which induces a given splitting $\mathcal{C}$ to the unique node or edge in $T'$ which induces this same splitting (and of course we map $\Omega$ to itself through the identity). We claim that these bijections are an isomorphism.

    To see this, we must characterize adjacency in these trees. Let $\mu$ be node in $T$, $\epsilon$ be an edge in $T$, and $\mathcal{C}_\mu$ and $\mathcal{C}_\epsilon$ be the splittings of $\Omega$ induced by each one (If $\mu$ happens to be a leaf node, set $\mathcal{C}_\mu = \set{\set{\mu}, \Omega \setminus \set{\mu}}$ for the purposes of this argument). We claim that $\mu$ connects to $\epsilon$ if and only if one of the two sectors of $\mathcal{C}_\epsilon$ is a also a sector of $\mathcal{C}_\mu$. The ``only if'' direction is straightforward, since if $\epsilon$ is an edge of $\mu$ then one of the connected components of $T \setminus \set{\mu}$ is the union of $\set{\epsilon}$ with one of the connected components of $T \setminus \set{\epsilon}$. The ``if'' direction now follows, since now we know each of the sectors of $\mathcal{C}_\epsilon$ is also in the splittings induced by $\epsilon$'s two neighbor nodes, so by Fact\ref{basic facts}(3), if $\mathcal{C}_\mu$ contains any sector of $\mathcal{C}_\epsilon$, then $\mu$ must be a node of $\epsilon$.

    It follows from this that node-edge adjacency is preserved by the previous bijection, so it is fact an isomorphism and we are finished.
    
\end{proof}

Notice that the finiteness condition in the previous theorem is necessary. For instance, no proper $D$-set can be the leaves of a (graph-theoretical) tree, since if such a tree $T$ existed, take a random leaf node $a$ and let $x$ be the inner node adjacent to it. Make a structural partition by partitioning the leaves of $T$ according to which connected component of $T \setminus \set{x}$ they lie on. Then $\set{a}$ is a component of this partition, but that is impossible since any node splitting in a proper $D$-set has infinitely many components.

\section{Homogeneity of D-sets with Predicates}

In this section we will characterize ultrahomogeneity in countable colored D-sets.

First, we make precise what we mean by a ``colored $D$-set.''

    \begin{definition}
        A \emph{coloring} of a $D$-set $\Omega$ is a finite set of predicates $P_i, \dots, P_n$ such that the sets $\set{P_1, \dots, P_n}$ form a partition of $\Omega$ into nonempty sets. If $x \in P_i \subseteq \Omega$ we say that $x$ has color $i$. A $D$-set with predicates which fulfill these conditions is called a \emph{colored $D$-set}. We denote the color of an element $x$ by $c(x)$.
    \end{definition}

\begin{lemma}\label{splitting extension lemma}
    Let $\Omega$ be a (not necessarily countable) $D$-set, $A$ be a substructure of $\Omega$ and $\mathcal{C}$ a splitting of $A$, then there exists a splitting $\mathcal{C}'$ of $\Omega$ which extends $\mathcal{C}$ in the sense that every sector of $\mathcal{C}$ is the intersection of a sector of $\mathcal{C}'$ with $A$. Furthermore if $\mathcal{C}$ is a node splitting then $\mathcal{C}'$ is unique, and if $\mathcal{C}$ is an edge splitting then $\mathcal{C}'$ may also be taken to be an edge splitting.
\end{lemma}
\begin{proof}
    We shall inductively construct $\mathcal{C}'$ element-by-element. Fix a $\kappa$-enumeration of $\Omega$, where $\kappa = |\Omega|$, and let $x$ be the first element of $\Omega \setminus \bigcup \mathcal{C}$. We shall either add $x$ to an existing sector of $\mathcal{C}$ or add $\set{x}$ as a new sector, according to the following rule: Designate as sector $\Sigma$ of $\mathcal{C}$ as \textit{suitable for $x$} (with respect to the splitting) if whenever $a \in \Sigma$ and $b,c \in A \setminus \Sigma$ we have $ax|bc$.
    \begin{itemize}
        \item If there exists a unique sector $\Sigma$ suitable for $x$, then add $x$ to the sector $\Sigma$.

        \item If no such sector exists add $\set{x}$ to $\mathcal{C}$ as a new sector. If this is the case (as we shall prove later) $\mathcal{C}$ is necessarily a node splitting.

        \item If there exist more than one sector suitable for $x$, then (as we shall prove later) $\mathcal{C}$ must be an edge splitting, and we may freely choose to add $x$ to either sector of $\mathcal{C}$, or add $\set{x}$ as a new sector.
    \end{itemize}

    We must check that this procedure always yields a splitting. We shall proceed by cases.

    \vspace{3pt}

    \textbf{Case 1:} $\mathcal{C}$ is a node splitting.

    \vspace{3pt}
    
    First, we shall note that that in this case, $\Sigma$ being suitable for $x$ is equivalent to $ax|bc$ holding for \textit{some} $a \in \Sigma, b,c  \in A \setminus\Sigma$ with $b,c$ not in the same sector. The ``only if'' direction is clear. For the ``if'' direction, suppose that $a\in \Sigma$ and $b, c$ in different sectors of $\mathcal{C} \setminus \set{\Sigma}$ are such that $ax|bc$ and let $a' \in \Sigma$ be arbitrary, then $aa'|bc$ and so by transitivity $a'x|bc$. Now if $b'$ is in the same sector as $b$ then
    \begin{equation*}
        \begin{split}
            xa'|bc &\Rightarrow b'a'|bc \vee xa'|b'c \quad \textrm{(former false since $bb'|ac$ instead)}\\
            &\Rightarrow xa'|b'c. \quad (*)
        \end{split}
    \end{equation*}
    and similarly
    \begin{equation*}
        \begin{split}
            xa'|bc &\Rightarrow b'a'|bc \vee xa'|bb' \quad \textrm{(former false since $bb'|ac$ instead)}\\
            &\Rightarrow xa'|bb'. \quad (**)
        \end{split}
    \end{equation*}
    Lastly, if $d$ is in a different sector from $a$, $b$ and $c$ then
    \begin{equation*}
        \begin{split}
            xa'|bc &\Rightarrow da'|bc \vee xa'|dc \quad \textrm{(former false as $\set{a',b,c,d}$ are in different sectors)}\\
            &\Rightarrow xa'|dc. \quad (***)
        \end{split}
    \end{equation*}
    By applying $(*)$, $(**)$ and $(***)$ repeatedly, we may replace $b$ and $c$ with any elements not in $\Sigma$, so that $xa'|b'c'$ for any $a' \in \Sigma$ and $b', c' \in A\setminus \Sigma$ and $\Sigma$ is indeed suitable for $x$.

    Now, let us see that the third case in the construction is impossible when $\mathcal{C}$ is a node splitting. Suppose there are two sectors $\Sigma_1, \Sigma_2$ both suitable for $x$, and let $\Sigma_3$ be a third sector. Take arbitrary elements $a \in \Sigma_1$, $b \in \Sigma_2$, $c\in \Sigma_3$. Since $\Sigma_1$ is suitable for $x$ we have $ax|bc$ and since $\Sigma_2$ is suitable for $x$ we have $bx|ac$, but $ax|bc$ and $bx|ac$ cannot be simultaneously true. We conclude that if a suitable sector exists, it is unique and so the third case in the construction is impossible.

    Notice that when $x$ is added to a suitable sector $\Sigma$ then the first condition in Definition \ref{splitting} is automatic. To check the second condition let $a,b,c \not \in \Sigma$ be in three different sectors of $\mathcal{C} \setminus \Sigma$. If some $D$-relation holds on $\set{x,a,b,c}$, say $xa|bc$, then by the equivalence at the start of this case the sector which contains $a$ is suitable for $x$, which contradicts the uniqueness of $\Sigma$, so no such $D$-relation may hold. We conclude that if $x$ is added to $\Sigma$, the result is a splitting.
    
    When $\set{x}$ is added as a new sector, checking the first condition in Definition \ref{splitting} reduces to checking that $aa'|bx$ for any $a, a',b$ such that $a$ and $a'$ share a sector of $\mathcal{C}$ and $b$ is in some other sector. To show this, pick $c$ in a third sector distinct from the sectors of $a$ and $b$, and note that by (D3),
    \begin{equation*}
            aa'|bc \Rightarrow ax|bc \vee aa'|bx .
    \end{equation*}
    But $ax|bc$ contradicts the assumption that no suitable sector exists, by the equivalence at the start of this section, so we conclude that $aa'|bx$, as we wanted.

    Lastly, we must check that the second condition in Definition \ref{splitting} holds when $\set{x}$ is added as a new sector. Just as when $x$ was added to an existing sector, if $a, b, c$ belong to different sectors of $\mathcal{C}$ and some $D$-relation holds between them, such as $ax|bc$, then the sector containing $a$ would be suitable for $x$, a contradiction. We conclude that when $\mathcal{C}$ is a node splitting, this procedure always yields a splitting.

    \vspace{3pt}

    \textbf{Case 2:} $\mathcal{C} = \set{\Sigma_1, \Sigma_2}$ is an edge splitting.

    \vspace{3pt}

    First we shall check that the second case in the construction is impossible, i.e. there is always a sector suitable for $x$. Suppose, to the contrary, that neither $\Sigma_1$ nor $\Sigma_2$ is suitable for $x$. Then this is witnessed by elements $a, b', c' \in \Sigma_1$ and $a', b, c \in \Sigma_2$ such that

    \begin{equation}
        \neg D(a, x; b, c)
    \end{equation}

    and

    \begin{equation}
        \neg D(a', x; b', c').
    \end{equation}

    Let $A_0 = \{a, a', b, b', c, c'\}$ and note that $\{\Sigma_1 \cap A_0, \Sigma_2 \cap A_0 \}$ is a splitting of the substructure $A_0$. Hence $a b' | bc$ since they are in different sectors, and by (D3), $(ax | bc) \vee (a b' | xc)$; but the former contradicts (1), and so

    \begin{equation}
        ab' | x c.
    \end{equation}
    
    Similarly, we have $a c' | bc$, hence $(ax | bc) \vee (ac' | xc)$, and the former is again ruled out by (1), so that

    \begin{equation}
        a c' | xc.
    \end{equation}

    By (3), (4), and transitivity, we conclude that $b' c' | xc$. But we also have that $b'c' | a' c$, so by transitivity again, $b' c' | a' x$. This contradicts (2), and so one of the two sectors in $\mathcal{C}$ must be suitable for $x$.

    Now let us check that the procedure described indeed yields a splitting. Since $\mathcal{C}$ is an edge splitting, the second condition for splittings becomes vacuous regardless of whether $x$ is added to an existing sector or as a new sector, so we need only check the first condition, which is automatic when $x$ is added to an existing sector, so that we need only check it when $\set{x}$ is added as a new sector. Again, this reduces to checking that $aa'|xb$ for any $a, a',b$ such that $a$ and $a'$ share a sector of $\mathcal{C}$ and $b$ is in the other sector, but this follows from the fact that both sectors of $\mathcal{C}$ are suitable for $x$ since this is the only way that $\set{x}$ can be added as a new sector. We conclude that when $\mathcal{C}$ is an edge splitting, this construction always yields a splitting

    \vspace{3pt}
    
    Repeat this construction inductively in the natural way to extend $\mathcal{C}$ to a full splitting $\mathcal{C}'$ of $\Omega$. Notice that if $\mathcal{C} = \set{\Sigma_1, \Sigma_2}$ is an edge splitting, then for any $a,a' \in \Sigma_1, b,b' \in \Sigma_2$ we have $aa'|bb'$ and so by (D3) either $ax|bb'$ or $aa'|bx$ (or both). And by an argument analogous to the previous, if $ax|bb'$ for this particular choice of $a,b,b'$, then it holds for any choice of $a \in \Sigma_1, b,b' \in \Sigma_2$, and analogously for $aa'|bx$. Then if $\mathcal{C}$ is an edge splitting, the second case in the construction can never occur, and so we may always choose to add $x$ to an existing sector and so $\mathcal{C}'$ can always be chosen to be an edge splitting.

    To finish the proof, we must check the uniqueness of $\mathcal{C}'$ when $\mathcal{C}$ is a node splitting. Let $\mathcal{C}'$ and $\mathcal{C}''$ be two extensions of a node splitting $\mathcal{C}$. Fix a $\Sigma\in \mathcal{C}$ and let $\Sigma' \in \mathcal{C}', \Sigma'' \in \mathcal{C}''$ be such that $\Sigma = \Sigma' \cap A = \Sigma'' \cap A$. If $x \in \Sigma' \setminus \Sigma$, then pick an $a \in \Sigma$ and $b,c \in A \setminus \Sigma$ such that $b,c$ are in different sectors (for this we use $|\mathcal{C}| \geq 3$), then we have that $xa|bc$ and $a,b,c$ are in different $\mathcal{C}$-sectors, and because $\mathcal{C}''$ extends $\mathcal{C}$ then $a,b,c$ are also in different $\mathcal{C}''$-sectors, so then $x$ must be in the same sector of $\mathcal{C}''$ as $a$. That is to say $x \in \Sigma''$. Thus $\Sigma' = \Sigma''$, but then by Fact \ref{basic facts}(3) $\mathcal{C}' = \mathcal{C}''$ and we are done.

\end{proof}

\begin{lemma}\label{type definability lemma}
    Suppose $A \cup \set{e}$ is a finite subset of a $D$-set with $e \not\in A$. Then there is a subset $A_0 \subseteq A$ such that $|A_0| \leq 3$ and $\qftp(e / A)$ is definable over $A_0$, where $\qftp$ stands for the quantifier-free type. 
    
    Moreover, when the splitting of $A$ induced by $e$ is a node splitting then $A_0$ can be any set of three elements from distinct sectors, and when it is an edge splitting then $A_0$ can be any set containing a complementary pair from one sector and any element from the other sector.
    
    In case the theory eliminates quantifiers, we can replace $\qftp$ with $\tp$.
\end{lemma}
\begin{proof}
    It suffices to find a definition for $D$-relations of the form $D(xy_1; y_2y_3)$.
    
    Let $\mathcal{C}$ be the splitting of $A$ induced by $e$. If $\mathcal{C}$ is a node splitting then let $a,b,c$ be elements in three distinct sectors of $\mathcal{C}$ and take $A_0 = \set{a, b, c}$. Notice that two elements $x,y \in A$ being in the same sector of $\mathcal{C}$ is a relation definable over $A_0$, by the formula
    $$xy|ab \vee xy|ac \vee xy |bc.$$
    On the other hand, when $\mathcal{C}$ is an edge splitting, if we take a pair $(b_1, b_2)$ of complementary elements from a sector of $\mathcal{C}$ and an $a$ outside that sector, and take $A_0 = \set{a, b_1, b_2}$, then which sector of $\mathcal{C}$ a given element belongs to is definable over $A_0$, since an $x$ belongs to the sector containing $a$ if and only if $D(b_1b_2, ax)$.
    
    We conclude that two elements sharing a sector of $\mathcal{C}$ is always $A_0$-definable for an $A_0$ of size at most $3$. Hence, to determine when $D(xy_1, y_2y_3) \in \qftp(e/A)$ with a formula in $y_1, y_2, y_3$, we may split into cases depending on how $y_1, y_2,$ and $y_3$ share sectors of $\mathcal{C}$.
    \begin{itemize}
        \item Case 1: $y_1$ does not share a sector of $\mathcal{C}$ with $y_2$ nor $y_3$:

        Then $D(e y_1; y_2 y_3)$ holds if and only if $y_2$ and $y_3$ share a sector, which is definable over $A_0$.

        \item Case 2: $y_1$ shares a sector with exactly one of $y_2$ or $y_3$:

        Then $D(e y_1; y_2y_3)$ never holds.

        \item Case 3: $y_1, y_2$ and $y_3$ are all in the same sector.

        Then $D(e y_1; y_2y_3)$ holds if and only if $D(x y_1; y_2y_3)$ for any $x$ outside their shared sector, and one of $a,b,c$ can always (definably) take the place of $x$ for this purpose.
    \end{itemize}
    Therefore $\qftp(e/A)$ is $A_0$-definable.

\end{proof}

\begin{remark}\label{spltting is determined by qftp}
    If $A \cup \set{e}$ is a $D$-set with $e \not\in A$, possibly a substructure of a larger $D$-set $\Omega$, then the splitting $\mathcal{C}$ of $A$ induced by $e$ is uniquely determined by $\qftp(e/A)$.
\end{remark}
\begin{proof}
    Let $e'$ have the same quantifier-free type over $A$ as $e$, and let  $a,b, c \in A$. If $a,b$ share a sector of $\mathcal{C}$ and $c$ is in a different sector, then $ab|ce'$ and if $a,b$ do not share a sector then $\lnot ab|ce'$. This proves $e'$ also induces $\mathcal{C}$.

    On the other hand, if $e'$ induces the same splitting as $e$, then $id|_A \cup \set{e,e'}$ is a (possibly partial) isomorphism that fixes $A$, so that $\qftp(e'/A) = \qftp(e/A)$.
\end{proof}

Now we will prove Theorem 1.

    \begin{proof}
        For the ``if'' direction, suppose $\Omega$ is a countable dense proper regular $D$-set. Let $A, B$ be two finite substructures of $\Omega$ and let $\phi:A \rightarrow B$ an automorphism between them. We shall extend $\phi$ to a full automorphism of $\Omega$ in a standard back-and-forth fashion. We shall describe only the first step of the construction, for all subsequent steps are analogous.

    Let $x$ be the first element of $\Omega$ (according to some fixed enumeration) which is not in $A$. Suppose that $x \in P_i$ and let $\mathcal{C}$ be the splitting of $A$ induced by the element $x$. Let $\mathcal{D}$ be the partition of $B$ which consists of the images of each sector of $\mathcal{C}$ under $\phi$. Clearly $\mathcal{D}$ is a splitting of $B$. We must split into cases based on the number of sectors in $\mathcal{C}$.

    \begin{itemize}
        \item Case 1: $\mathcal{C}$ is an edge splitting.

        Write $\mathcal{C} = \{C_1, C_2\}$ and $\mathcal{D} = \{D_1, D_2\}$. Let $a,b$ be a pair of complementary elements of $C_1$ and $c,d$ be a pair of complementary elements in $C_2$, as in Lemma \ref{BAF lemma 2}. Since $\phi$ is an isomorphism, $\phi(a)$ and $\phi(b)$ are complementary in $D_1$ and $\phi(c)$ and $\phi(d)$ are complementary in $D_2$. Since $\phi(a)\phi(b)|\phi(c)\phi(d)$, by density there exists a $y_0 \in \Omega$ such that $y_0 \phi(b)|\phi(c)\phi(d)$, $\phi(a)y_0|\phi(c)\phi(d)$, $\phi(a)\phi(b)|y_0\phi(d)$ and $\phi(a)\phi(b)|\phi(c)y_0$. Note that $\set{\set{\phi(a),\phi(b)}, \set{\phi(c),\phi(d)}, \set{y_0}}$ is a splitting of $\set{\phi(a),\phi(b),\phi(c),\phi(d),y_0}$, and so by Theorem \ref{D set extension lemma} this extends to a node splitting $\mathcal{D}'$ of all $\Omega$. Let $\Sigma_0 \in \mathcal{D}'$ be the sector which contains $y_0$, then by hypothesis there is some $y \in \Sigma_0 \cap P_i$. Since $y$ is in the same $\mathcal{D}'$-sector as $y_0$, we have $y\phi(b)|\phi(c)\phi(d)$, $\phi(a)y|\phi(c)\phi(d)$, $\phi(a)\phi(b)|y\phi(d)$ and $\phi(a)\phi(b)|\phi(c)y$.
        
        By Lemma \ref{type definability lemma} and Remark \ref{spltting is determined by qftp}, $y$ induces the splitting $\mathcal{D}$ on $B$, and because $x$ induces $\mathcal{C}$ on $A$ it is clear that $\phi \cup \set{(x,y)}$ is an isomorphism, as desired.

        \item Case 2: $\mathcal{C}$ is a node splitting.

        Let $d = |\mathcal{C}| = |\mathcal{D}|$. Because $x$ induces the splitting $\mathcal{C}$, we have that $\mathcal{C} \cup \set{\{x\}}$ is in itself a splitting of $A \cup \set{x}$. By Lemma \ref{splitting extension lemma} $\mathcal{C} \cup \set{\{x\}}$ extends to a full splitting $\mathcal{C}'$ of $\Omega$ with at least $d+1$ sectors. Now let $\mathcal{D}'$ be the extension of $\mathcal{D}$ to $\Omega$, again as in Lemma \ref{splitting extension lemma}. Then by regularity $|\mathcal{C}'| = |\mathcal{D}'| > d$ and so there exists a sector $D \in \mathcal{D}'$ such that $D$ is disjoint from every sector of $\mathcal{D}$. By hypothesis, $D \cap P_i \neq \emptyset$ so take $y \in D \cap P_i$. 
        
        As before, by Lemma \ref{type definability lemma} and Remark \ref{spltting is determined by qftp}, $y$ induces $\mathcal{D}$ and so $\phi \cup \set{(x,y)}$ is an isomorphism, as desired.
    \end{itemize}
    By continuing this process inductively, switching the roles of $A$ and $B$ at every step, we may extend $\phi$ to a full automorphism of $\Omega$, and so $\Omega$ is ultrahomogeneous.
    
    Now for the ``only if'' direction. Suppose first that there is a $P_i$ and an infinite sector $\Sigma$ such that $P_i \cap \Sigma = \emptyset$. Take an arbitrary $b_0 \in P_i$ and $a_1, a_2 \in \Sigma$ of the same color, which exist by the pigeonhole principle. By propriety there exists an $a_3$ such that $b_0a_1|a_2a_3$.

    Consider the partial isomorphism $a_1 \mapsto a_2, a_2 \mapsto a_1, a_3 \mapsto a_3$. This is a partial isomorphism since $a_1, a_2$ were chosen of the same color. We claim this cannot be extended to a full automorphism as $b_0$ cannot have an image. If it did have an image $c_0$, then $P_i(c_0)$ and $c_0a_2|a_1a_3$ hold. Then by (D3) $b_0a_2|a_1a_3$ or $c_0a_2|a_1b_0$, but the former is impossible since $b_0a_1|a_2a_3$ holds instead and the latter too since $b_0c_0|a_1a_2$ holds instead (since $a_1,a_2 \in \Sigma$, $b_0\not\in \Sigma$, and $c_0 \notin \Sigma$ since $P_i(c_0)$ holds). Thus the partial isomorphism cannot be extended.

    Thus we may assume that $P_i \cap \Sigma$ is nonempty for every sector $\Sigma$ and every color $i$, and now suppose that $\Omega$ is not dense. Thus there are $a,b,c,d$ such that $ab|cd$ but there is no $e$ such that $ab|ed$ and $ae|cd$. Notice that $a,b,c,d$ must be all distinct, since otherwise propriety implies the existence of such an $e$. Extend the splitting $\set{\set{b},\set{c},\set{d}}$ to a node splitting of $\Omega$, and let $x$ be an arbitrary element of the sector to which $d$ belongs such that $x$ and $c$ have the same color. Then $bc|dx$, which implies that $\lnot(xb|cd)$ so that $x \neq a$. Thus by applying (D3) to $(x,b,c,d)$ we may also obtain that $ab|cx$, and by applying (D3) to $(b,c,x,d)$ we may deduce that $ac|xd$. 
    
    Now, notice that the map $a \mapsto a, b\mapsto b, d\mapsto d$ and $x\mapsto c$ is a partial isomorphism since it preserves the only nontrivial $D$-relation among these points and since $x$ and $c$ have the same color. However, this map cannot be extended to a full automorphism since for $(a,b,x,d)$ we have that $c$ fulfills $ab|cd$ and $ac|xd$ but by construction no possible image of $c$ could exist which fulfills this property on $(a,b,c,d)$.

    Lastly, suppose towards a contradiction that $\Omega$ is not regular, and let $\Sigma_1, \Sigma_2$ be two node splittings with a different number of sectors (say $|\Sigma_1| > |\Sigma_2| = n$). Take a tuple of elements $(a_1, \cdots, a_n)$ such that each is from a different sector of $\Sigma_2$ and each belongs to $P_1$. Also take $n+1$ elements $(b_1, \cdots, b_n, b_{n+1})$ from $n+1$ different sectors in $\Sigma_1$ and which all belong to $P_1$. The map $b_i \mapsto a_i$ is a partial isomorphism, but it cannot be extended to a full automorphism since if $a_{n+1}$ was the image of $b_{n+1}$ under such an automorphism, it would have to be in some sector of $\mathcal{C}$, say the one that contains $a_i$, and thus $a_{n+1}a_i|a_ja_k$ for any $j, k \leq n$ distinct from each other and $i$, but this contradicts the fact that $b_{n+1}b_i |b_jb_k$ does not hold.
        
    \end{proof}

    There is a weaker notion of homogeneity which is nevertheless interesting to this structure:

    \begin{definition}
        A structure $\Omega$ with a predicate $P$ which is not equal to $\emptyset$ nor $\Omega$ is called $P$-homogeneous if whenever $A, B$ are finitely generated substructures such that $A \cap P, A \setminus P, B \cap P$ and  $B \setminus P$ are all nonempty and $\phi: A \rightarrow B$ is a partial isomorphism between them, then $\phi$ may be extended to an automorphism of $\Omega$.
    \end{definition}

\begin{lemma}\label{true edge witness property}
    Let $\set{C_1, C_2}$ be an edge splitting of a $D$-set $\Omega$, then $\set{C_1, C_2}$ is a true edge splitting if and only if for all $a,b \in C_1$, $c,d \in C_2$, there exists witnesses of density $y_1, y_2$ for $(a,b,c,d)$ such that $y \in C_1$ and $y_2 \in C_2$.
\end{lemma}
\begin{proof}
    Suppose $\set{C_1, C_2}$ is a true edge splitting. Notice that $\set{\set{a}, \set{b}, \set{c, d}}$ is a node splitting of $\set{a,b,c,d}$. Extend it to a node splitting $\mathcal{D}$ of all $\Omega$. By the One Sector Lemma (Lemma \ref{the one sector lemma}) all but one of the sectors of $\mathcal{D}$ must be contained in $C_1$ or $C_2$ and the exception contains the remaining one, but two sectors of $\mathcal{C}$ contain elements of $C_1$, so it must be that all but one of the sectors of $\mathcal{D}$ are contained specifically in $C_1$, and the exception is the sector $D_0 \in \mathcal{D}''$ containing $c$ and $d$, so then $D_0 \supseteq C_2$, and this containment must be strict since otherwise $\mathcal{D}$ would refine $\set{C_1, C_2}$ contradicting the fact that it is a true node splitting. Let $y_1$ be an element of $C_1 \cap D_0$. Notice that $yb|cd$, $ay|cd$, $ab|yd$ and $ab|cy$ using the fact that $\mathcal{D}$ and $\set{C_1, \lnot C_2}$ are both splittings as needed. That is to say $y_1$ is  a witness of density for $(a, b, c, d)$ which is in $C_1$, as we required. The proof for $y_2$ is analogous.

    Now suppose for some $a,b \in C_1, c,d \in C_2$, $y_1$ does not exist, that is al witnesses of density for $(a,b,c,d)$, if they exist, lie in $C_2$. Extend $\set{\set{a,b}, \set{c}, \set{d}}$ to a node splitting $\mathcal{D}'$ of all $\Omega$. Again by the One Sector Lemma all but one of the sectors of $\mathcal{D}$ must be contained in $C_2$. Let $D_0 \in \mathcal{D}'$ be the sector containing $\set{a,b}$, then $D_0$ is the one sector not contained in $C_2$, and so $D_0 \supseteq C_1$, but notice that if $x \in C_1 \setminus D_0$ then $x$ would be a witness of density for $(a,b,c,d)$ and so such $x$ cannot exist and $D_0 = C_1$. Thus $\mathcal{D}'$ refines $\set{C_1, C_2}$ and $\set{C_1, C_2}$ is not a true edge splitting.
\end{proof}

\begin{theorem}
        Let $\Omega$ be a \textbf{proper} countable $D$-set with a unary predicate $P$. $\Omega$ is $P$-homogeneous if, and only if, it satisfies the following properties:

\begin{enumerate}
    \item $\Omega$ is a dense $D$-set;
    \item Any two node splittings of $P$ (as a subset of $\Omega$) contain the same number of sectors (either a natural number greater than $2$, or $\aleph_0$), and analogously for $\Omega\setminus P$; and
    \item The set $P$ satisfies one of the following three conditions:
    \begin{enumerate}
        \item For every infinite sector $\Sigma$ of every splitting, $\Sigma \cap P \neq \emptyset$ and $\Sigma \setminus  P \neq \emptyset$;
        \item $\set{P, \Omega \setminus P}$ is a true edge splitting; or
        \item either $|P| = 1$ or $|\Omega \setminus P| = 1$.
    \end{enumerate}
\end{enumerate}
\end{theorem}

    \begin{proof} ~

    \textit{Sufficiency of $(1)+(2)+(3a)$}
    
    \vspace{3pt}
    
    If (3a) holds then notice that (2) becomes equivalent to any two node splittings of $\Omega$ having the same size, so by Theorem \ref{ultrahomogeneity theorem}, $\Omega$ is ultrahomogenous with $P$ and $\lnot P$ both nonempty and we are done. 

    \vspace{3pt}

    \textit{Sufficiency of $(1)+(2)+(3c)$}
    
    \vspace{3pt}
    
    If (3c) holds then the proof is identical to the ``if'' direction of Theorem \ref{ultrahomogeneity theorem} by noting that the initial partial isomorphism must fix the sole element of $P$ or $\Omega \setminus P$, as the case may be, and so $P$-homogeneity follows.

    \vspace{3pt}

    \textit{Sufficiency of $(1)+(2)+(3b)$}
    
    \vspace{3pt}

    Once again we will repeat the proof of the ``if'' direction of Theorem \ref{ultrahomogeneity theorem}, with modifications. Recall that $\phi: A \rightarrow B$ is a partial isomorphism of finite substructures, which for this proof are assumed to contain elements of $A$ and $B$, $x$ is an element to be added to the domain of $\phi$, which we will assume without loss of generality to be in $P$, $\mathcal{C}$ is the splitting of $A$ induced by $x$ and $\mathcal{D}$ is the image of $\mathcal{C}$ under $\phi$. 

    Before we start, note that $\set{A \cap P, A \setminus P}$ is a splitting of $A$, since it is the restriction to $A$ of the splitting $\set{P, \Omega \setminus P}$, and similarly $\set{B \cap P, B \setminus P}$ is a splitting of $B$. We wish to construct a $y$ such that $\phi \cup \set{(x,y)}$ is an isomorphism.

    When $\mathcal{C}$ is a node splitting, then the construction of $y$ in Theorem \ref{ultrahomogeneity theorem} relied on a slightly different form of condition (2), in which all node splittings of $\Omega$ had the same size. Here, rather all node splittings of $P$ (and of $\lnot P$) have the same size. With this in mind, modify the construction as follows:

    By construction, $\mathcal{C} \cup \set{\set{x}}$ and $\mathcal{D}$ are node splittings of $A \cup \set{x}$ and $B$, respectively, and so they may be extended to full node splittings $\mathcal{C}'$ and $\mathcal{D}'$ of $\Omega$. 
    
    By the One Sector Lemma all but one of the sectors of $\mathcal{C}'$ are contained in $P$ or $\lnot P$, and the exception contains the other. In particular, if two sectors have elements of $P$, then all but one sector is contained in $P$. The sector which contains $x$ and at least one other sector of $\mathcal{C}$ have this property, so all but one sectors of $\mathcal{C}$ are contained in $P$. Similarly all but one sectors of $\mathcal{D}$' are contained in some sector of $\set{P, \lnot P}$, and since there are at least two sectors in $\mathcal{D}'$ containing images of points of $B \cap P$ (because sectors of $\mathcal{D}$ are images of sectors of $\mathcal{C}$ and there are at least three sectors in $\mathcal{C}$, at least two contained in $P$), again that sector is $P$.

    So then all but one sectors of $\mathcal{C'}$ and $\mathcal{D}'$ are contained in $P$, and the exception in each case cannot be $\lnot P$ as otherwise $\set{P, \lnot P}$ could be refined to a node splitting, contradicting (3b), so that $|\mathcal{C}'| = |\mathcal{C}' \cap P|$ and $|\mathcal{D}'| = |\mathcal{D}' \cap P|$. Then $\mathcal{C}' \cap P$ and $\mathcal{D}' \cap P$ to $P$ are node splittings of $P$, and so by (2) they have the same size and particularly all sectors of $\mathcal{D}'$ have elements of $P$. As in the proof of Theorem \ref{ultrahomogeneity theorem}, choose a $y$ in a sector of $\mathcal{D}'$ which is disjoint from any sector of $\mathcal{D}$ and such that $y \in P$. By an analogous argument to that theorem's proof, $\phi \cup \set{(x,y)}$ is an isomorphism.

    Now, when $\mathcal{C}$ is an edge splitting, the construction once again proceeds analogously to Theorem \ref{ultrahomogeneity theorem}, but we must ensure that the constructed $y$ has the same $P$-type as $x$.

    Recall that we set $\mathcal{C} = \set{C_1, C_2}, \mathcal{D} = \set{D_1, D_2}$ and we chose complimentary pairs of elements $a,b \in C_1, c,d\in C_2$ as in Lemma \ref{BAF lemma 2}. Note that one of $\set{a,b}$ or $\set{c,d}$ must be contained in $P$, as otherwise both $C_1$ and $C_2$ have elements of $\lnot P$, so all but one of the sectors in the splitting $\set{C_1, C_2, \set{x}}$ are contained in $\lnot P$, and the exception cannot be $\set{x}$, so that $x \not \in P$ in contradiction with the hypothesis.

    So, suppose that $\set{a,b} \subseteq P$. If $c \in P$ then $\phi(a), \phi(b), \phi(c)$ are all also in $P$ and so by an analogous argument to the previous, $y \in P$ and we are done. Similarly we are done if $d \in P$, so suppose that $c,d \not \in P$.

    Now, since $a,b \in P$ and $c,d \not \in P$, the same holds of their images under  $\phi$. That is to say we have $\phi(a), \phi(b) \in P, \phi(c), \phi(d) \not \in P$, then by Lemma \ref{true edge witness property} there exists a witness of density $y$ for $(\phi(a), \phi(b), \phi(c), \phi(d))$ which is in $P$, and so with this choice of $y$, $\phi\cup \set{(x,y)}$ is an isomorphism and we are done.

    This concludes the proof of the ``if'' direction of this theorem.

    \vspace{3pt}

    \textit{Necessity of (3)}

    \vspace{3pt}
    
    Suppose that (3) does not hold. That is, that there exists an infinite sector $\Sigma$ of some splitting which is contained in $P$ or disjoint from $P$ (without restriction, suppose the former), $P$ is not a true edge splitting, and $\lnot P$ has at least two elements (which are obviously outside $\Sigma$). As in the previous theorem, we wish to avoid a case where $P = \Sigma$, so if that is the case replace $\Sigma$ with another infinite sector which is strictly contained in it, so that $\Omega \setminus \Sigma$ can be assumed to contain elements of $P$ and $\lnot P$.

    We shall split the proof into two cases, depending on whether a certain finite substructure exists in $\Omega$:
    \begin{itemize}
        \item  Case 1: there exist elements $a, b \in P$, $c,d  \in \lnot P$ all distinct such that $ac|bd$.

        If this is the case, then let $x,y \in \Sigma$ be distinct and consider the partial isomorphism given by $a \mapsto x,  b\mapsto y, d\mapsto d$. This is an isomorphism since it preserves the $P$-types of the elements and its domain and range have elements of $P$ and $\lnot P$, but it cannot be extended to a full automorphism since $c$ can have no image. If $c'$ was such an image, then since $ac|bd$ we would have $xc'|yd$, so then $c' \in \Sigma$ since otherwise the relation could not hold as $xy|c'd$ would hold instead (because $x,y \in \Sigma, c', d \not\in \Sigma)$, but then $c' \in P$ and so it cannot be the image of $c$ since $c \in \lnot P$. Thus $\Omega$ is not P-homogenous.
    \end{itemize}
    \begin{itemize}

        \item  Case 2: there do not exist such elements.

        Define an equivalence relation $\sim$ on $\Omega$ as follows: $a \sim b$ if and only if, for some $Q \in \set{P, \lnot P}$, $a$ and $b$ are both in $Q$ and for any $x,y \not \in Q$ we have $xy|ab$. We must check that this is indeed an equivalence relation. Reflexivity and symmetry are trivial. For transitivity, let $a,b,c \in Q$ such that $a\sim b$ and $b \sim c$, and suppose for a contradiction there are $x,y \not \in  Q$ such that $\lnot ac|xy$. We know $bc|xy$ by hypothesis so by (D3) either $ac|xy$ or $bc|xa$. The former is assumed untrue so $bc|xa$. Similarly, we know $ba|xy$ and so by (D3) either $ca|xy$ or $ba|xc$ and the former is assumed untrue so $ba|xc$, but then $cb|ax$ and $ba|xc$ are simultaneously true which contradicts (D2), so that $\sim$ is indeed an equivalence relation.

        A few observations are pertinent about this relation: If $a,b \in Q$ and $a \not \sim b$, then if $x,y$ are witnesses of this (i.e. $x, y \not \in Q$ and $\lnot ab|xy$), then also $x\not \sim y$, as witnessed by $a,b$. Also, if $x,y \not \in Q$ witness $a \not \sim b$, then in fact any pair $x', y' \not\in Q$ such that $x' \not \sim y'$ will witness $a \not \sim b$. Too see this we will show that $x$ may be switched for any $\lnot Q \ni x' \not \sim y$ (as we may later repeat for $y'$ and obtain the full result). Let $c,d \in Q$ witness that $x' \not \sim y$, that is $\lnot cd|xy$. Then we have the following chain of implications by repeated applications of (D3):
        
        \begin{equation*}
            \begin{split}
                ab|x'y &\Rightarrow xb|x'y \vee ab|xy \quad \textrm{(latter false since $x,y$ witness $a \not\sim b$)}\\
                & \Rightarrow xb|x'y\\
                & \Rightarrow xc|x'y \vee xb|x'c \quad  \textrm{(latter false by starting assumption.)}\\
                & \Rightarrow xc|x'y \\
                & \Rightarrow dc|x'y \vee xc|x'd \quad \textrm{(latter false by starting assumption.)}\\
                & \Rightarrow cd|x'y
            \end{split}
        \end{equation*}
        
        Here ``starting assumption'' refers to the assumption at the start of this case, which is equivalent to the statement that $x_1x_2|x_3x_4$ is the only possible $D$-relation if $x_1, x_2 \in Q, x_3, x_4 \in \lnot Q$. Since we know $\lnot cd|xy$, we have $ab|x'y$ and so indeed $x', y$ witness $a \not \sim b$. As a last corollary, either this relation has two classes, $P$ and $\lnot P$, or it has at least 4 classes, at least 2 contained in $P$ and 2 in $\lnot P$. 

        We now claim that the classes of $\sim$ in fact form a splitting. To see the first condition, let $a \sim b$ and $c, d$ be such that $c\not \sim a, d \not \sim a$. If $a,b \in Q$ and $c,d \not \in Q$ then it is clear from the construction that $ab|cd$, and we are done. If $a,b,c \in Q, d\not \in Q$ then let $e \not\in Q$ be such that $e \not \sim d$, then $ab|ed$ and so by (D3) $cb|ed$ or $ab|cd$. The former contradicts $b\not \sim c$ so $ab|cd$ and we are done. Lastly, if $a,b,c,d \in Q$, let $e, f \in \lnot Q, e \not \sim f$ and repeat the previous argument twice on $ab|ef$ to obtain $ab|cd$ as desired. Thus the first condition is fulfilled.

        Now for the second condition. Take four elements $a,b,c,d$ in different $\sim$-classes. Up to relabeling there exist three cases: (1) $a,b \in Q, c,d \not \in Q$, (2) $a,b,c \in Q$, $d \not \in Q$ or (3) $a,b,c,d \in Q$ for some $Q$. But note that for a tuple in case (3) if we apply (D3) to the $D$-relation between with an element in $\lnot Q$ we may obtain a new tuple of 4 elements with a $D$-relation between them, this time in case (2), and then for a tuple in case (2) we may repeat the trick and once again use (D3) with an element in $\lnot Q$ (which is not in the same class as $d$ and such that it replaces an element of $Q$ in either disjunct of (D3)) so that we obtain a 4-tuple in case (1). Thus, we may well assume that $a,b \in Q$ and $c,d \in \lnot Q$ and that there is a $D$-relation between them, but this is impossible, as the only possible such relation is $ab|cd$ by the starting assumption, and this would contradict the fact that $c,d$ must witness $a \not \sim b$ by the earlier observation (and vice versa), and so indeed the classes of $\sim$ forms a splitting $\mathcal{C}$.

        This constructed splitting has the property that it refines the partition $\set{P, \lnot P}$. In fact it may be equal to it. However, since $\set{P, \lnot P}$ is not a true edge splitting, if $\mathcal{C} = \set{P, \lnot P}$ then then there is a refinement of $\mathcal{C}$ which is also a splitting, so that we may assume that $\mathcal{C}$ is a strict refinement of $\set{P, \lnot P}$. Let $Q \in \set{P, \lnot P}$ be such that it contains two or more sectors $\Sigma_1, \Sigma_2 \in \mathcal{C}$, and let $\Sigma_0 \in \mathcal{C}$ be contained in $\lnot Q$. 
        
        Pick $a \in \Sigma_0$, $b,c \in \Sigma_1$, $d \in \Sigma_2$. and consider the partial isomorphism $a \mapsto a, b \mapsto b, c\mapsto d$. It is an isomorphism since it preserves the $P$-types and its domain and range have elements of $P$ and $\lnot P$, but it cannot be extended to a full automorphism since $d$ can't have an image: since $bc|ad$ (as $b,c \in \Sigma_1, a,d\not \in \Sigma_1$), such an image $d'$ would satisfy $bd|ad'$, but then $d'$ cannot be in $\lnot Q$ since $d \in Q$, it cannot be in $\Sigma_2$ as then $dd'|ab$ would hold instead, it cannot be in $\Sigma_1$ as $bd'|ad$ would hold instead, and it cannot be in $Q \setminus (\Sigma_1 \cup \Sigma_2)$ because then no $D$-relations could hold on $(a,b,d,d')$. Thus $\Omega$ is not P-homogeneous and we are done.
        \end{itemize}

    This shows the necessity of condition (3), so to prove the necessity of (1) and (2) we will assume without loss of generality that (3) holds. The case in which (3c) holds reduces easily to the previous theorem. The case of (3a) also reduces to the previous theorem by noting that all the partial isomorphisms in the analogous section can be chosen to have elements of $P$ and $\lnot P$ in their domain and range, so we need only check the necessity of (1) and (2) in the case of (3b).

    \vspace{3pt}

    \textit{Necessity of (1) (assuming (3b))}

    \vspace{3pt}

    Since $\set{P, \lnot P}$ is an edge splitting, if $a,b \in P$ and $c,d \in \lnot P$, then $ab|cd$. In particular by a straightforward check this implies that if $ab|cd$ and there are elements of $P$ of each side on the line (i.e ($Pa \vee Pb) \wedge (Pc \vee Pd)$ holds), then any witness of density for $(a,b,c,d)$ is in $P$

    Now, as in the proof of Theorem \ref{ultrahomogeneity theorem}, suppose there exist $a,b,c,d$ such that $ab|cd$ but they have no witness of density. If all are in $P$ (or $\lnot P$), then repeat the proof of \ref{ultrahomogeneity theorem} with the addition of fixing a dummy element $y \not \in P$ in the isomorphism. Because of the previous observation this isomorphism cannot be extended to a full automorphism by an unchanged argument.

    If three of $a,b,c,d$ are in $P$, then relabel so that $a \in \lnot P$ and repeat the argument in Theorem \ref{ultrahomogeneity theorem}, by doing this we ensure that the element $x$ constructed therein is in $P$ and so the rest of the argument proceeds identically (no need to fix a dummy element).

    If two of $a,b,c,d$ are in $P$, then either they are $a$ and $b$ or they are $c$ and $d$, but then by Lemma \ref{true edge witness property} it is impossible for them not to have a witness of density, so we are done.

    \vspace{3pt}

    \textit{Necessity of (2) (assuming (3b))}

    \vspace{3pt}

    Suppose $\mathcal{C}$, $\mathcal{D}$ are two node splittings of $P$ with a different number of elements, say $|\mathcal{C}|>|\mathcal{D}|$. Extend both to to splittings $\mathcal{C}'$, $\mathcal{D}'$ of $\Omega$. Once again by the One Sector Lemma all but one sectors of both $\mathcal{C}'$ and $\mathcal{D}'$ are contained in $P$, so that there are distinguished sectors $C_0 \in \mathcal{C}', D_0 \in \mathcal{D}'$ which intersect $\lnot P$.

    Take an $a_0 \in C_0 \setminus P, b_0 \in D_0 \setminus P$, and a sequence $b_1, \dots, b_n$ from each sector of $\mathcal{D}$ (not including the one which extends to $D_0$, if present) and an analogous sequence $a_1, \dots, a_n, a_{n+1}$ from different sectors of  $\mathcal{C}$, as in the proof of the necessity of (2) in Theorem \ref{ultrahomogeneity theorem}. By a similar argument to that proof the isomorphism $b_i \mapsto a_i$ does not extend to a full automorphism and its domain includes elements of $P$ and $\lnot P$, so we are done.

    \begin{corollary}\label{Quantifier Elimination}
    The theory of a dense regular $D$-set has quantifier elimination.
\end{corollary}
\begin{proof}
    Any countable model of the theory is an $\omega$-categorical relational structure whose countable structure is ultrahomogeneous, hence its theory has has quantifier elimination.
\end{proof}

    \end{proof}

\section{Indiscernibles, distality, and dp-minimality}

In this section we will characterize indiscernible sequences in colored $D$-sets with quantifier elimination, which we will use to prove their dp-minimality (Theorem~\ref{dpmin}). An important part of our analysis is to understand when an $\emptyset$-indiscernible sequence $s$ is indiscernible over a larger parameter set $B$, for which we define its  ``hull'' $H(s)$ (Definition~\ref{disc_hull}) and show that $s$ is indiscernible over $B$ if, and only if, $H(s)$ is disjoint from $B$ (Theorem~\ref{Indiscernibility over a set full characterization}). More or less, this hull $H(s)$ is the union of all the sectors which intersect the moving parts of the sequence $s$. We also show that mutually indiscernible sequences have disjoint hulls (Corollary~\ref{disjoint_hulls}), and from this, dp-minimality quickly follows. This overall strategy for establishing dp-minimality was inspired by similar work by Javier de la Nuez González on planar graphs with separation predicates (see \cite{nuezgraphs}). We also note that in recent work by Almazaydeh, Braunfeld, and Macpherson, indiscernible sequences in certain \emph{limits of $D$-sets} were analyzed (\cite{almazaydeh2024omega}), but to our knowledge this is the first time a criterion for indiscernibility in a pure $D$-set has been given.

First we recall some standard definitions.

\begin{definition}
    Let $I$ be a linear order. An ordered sequence $(a_i: i \in I)$ of tuples in an $L$-structure $M$ is called \emph{(order) indiscernible} if for any formula $\phi$ and any two strictly increasing sequences $i_0 < \dots < i_n, j_0 < ... < j_n$ in $I$, we have $M \models \phi(a_{i_0}, \dots, a_{i_n})$ if and only if $M \models \phi(a_{j_0}, \dots, a_{j_n})$.
\end{definition}

\begin{definition}
    A theory $T$ is distal if whenever $s$ is an indiscernible sequence of singletons over a set $B$ that is the concatenation $s_1 + s_2$ of two indiscernible sequences without endpoints, and $c$ is some element such that $s_1 + c + a_2$ is indiscernible over the empty set, then $s_1 + c + s_2$ is indiscernible over $B$.
\end{definition}

\begin{definition}
    A theory $T$ is dp-minimal if for every model $\Omega$ of $T$, every pair $s_1, s_2 \subseteq \Omega$ of mutually indiscernible sequences and any (singleton) element $b \in \Omega$, one of $\set{s_1, s_2}$ is indiscernible over $b$.
\end{definition}

This definition coincides with other definitions of dp-minimalility. See \cite{simon2015guide} for details.

In order to avoid troublesome edge cases that appear in results in this section whenever indiscernible sequences have a first or last element, we shall assume unless otherwise noted that all indiscernible sequences are unbounded on both sides (henceforth ``unbounded sequences''). Since any indiscernible sequence can be extended to an unbounded indiscernible sequence in a sufficiently saturated extension, it suffices to check dp-minimality for unbounded sequences, and distality only requires checking such sequences.

\begin{lemma}\label{Indiscernibles in one variable}
    Let $\Omega$ be a dense, regular $D$-set. A sequence $s = (a_i: i \in I)$ of singletons is order-indiscernible if, and only if, one of the following three conditions hold:
    \begin{enumerate}
        \item $s$ is constant; or
    
        \item All elements of $s$ are distinct and there are no non-trivial $D$-relations between elements of $s$; or

        \item All elements of $s$ are distinct and for any $i < j < k < \ell$, we have $D(a_i a_j; a_k a_\ell$).
    \end{enumerate}
\end{lemma}

Before we prove this lemma, we introduce the following terminology:
\begin{definition}
\label{petalmonotonic}
    A sequence satisfying condition (2) in the previous lemma is called a \emph{petaled} sequence, and a sequence satisfying condition (3) is called a \emph{monotonic} sequence.
\end{definition}

\begin{proof}[Proof (of Lemma \ref{Indiscernibles in one variable}).]
    That a constant sequence is indiscernible is trivial. In a monotonic sequence the truth value of $D(a_ia_j;a_ka_\ell)$ depends only on the order type of $(i, j, k, \ell)$, since we know it's true for $i<j<k<\ell$ and other order types follow from axioms (D1) and (D2). For petaled sequences the truth value of $D(a_ia_j;a_ka_\ell)$ is always false, and thus can also be said to be determined by the order type of $(i, j, k, \ell)$. Since we have assumed density and regularity, we have quantifier elimination by Corollary \ref{Quantifier Elimination}. That is any formula is a boolean combination of $D$-relations and equalities, so indeed the truth value of any formula with parameters in a petaled or monotonic $s$ is determined by the order type of its parameters and so such sequences are indiscernible.
    
    Thus conditions (1), (2) and (3) are each sufficient for indiscernibility. Let us now check their necessity. Suppose $s$ is indiscernible and assume that $s$ is not constant, It is easy to see that all elements of $s$ are thus distinct. The truth value of $D(a_i a_j; a_k a_\ell)$ is thus determined by the order type of $(i,j,k,\ell)$. We make the notational shorthand of writing $ij|k\ell$ for $D(a_i a_j; a_ka_\ell)$. We have the following cases:
    \begin{itemize}
        \item Case 1: $ij|k\ell$ holds when $i < j < k < \ell$

        Then the sequence is monotonic

        \item Case 2: $ik|j\ell$ holds when $i < j < k < \ell$

        We claim this is impossible. For this, pick $i < j < k < \ell < m$, then we would have $ik|j\ell$, so by (D3) either $im|j\ell$ or $ik|m\ell$. However the former conflicts with $i\ell|jm$ and the latter conflicts with  $i\ell|km$, both of which follow from the order $i<j<k<\ell<m$.

        \item Case 3: $i\ell|jk$ holds when $i < j < k < \ell$

        Again we claim this to be impossible. Pick $i < j < k < \ell < m$ as before. Then we would have $i\ell|jk$ and thus by an alternate form of (D3) either $m\ell|jk$ or $i\ell|mk$. The former conflicts with $jm|kl$ and the latter with $im|k\ell$, both of which follow from the order $i<j<k<\ell<m$.

        \item Case 4: There are no non-trivial $D$-relations.

        Then the sequence is petaled.
    \end{itemize}
    Thanks to axiom (D1), these are all the cases. Thus the conditions are necessary for indiscernibility.
\end{proof}

\begin{remark}
    The ``necessity'' portion of the previous proof holds even in the absense in quantifier elimination. Thus, even in a general $D$-set all indiscernible sequences of singletons are either constant, monotonic or petaled.
\end{remark}

\begin{definition}
\label{disc_hull}
    Given an 1-element indiscernible sequence $s = (a_i)_{i \in I}$, we define $H(s)$ the \emph{discernible hull} of $s$ as follows:
    \begin{itemize}
        \item If $s$ is constant, then $H(s)$ is empty.
        \item If $s$ is petaled, then extend $\set{\set{a_i}: i \in I}$ to a full node splitting $\mathcal{C}$. $H(s)$ is the union of all sectors of $\mathcal{C}$ which contain elements of $s$.
        \item If $s$ is monotonic, then $H(s)$ is the union of $s$ with the sets:
        \begin{equation*}
            \begin{split}
                H_1(s) &:= \set{x: \textrm{ there are } i < j < k  \textrm{ such that }a_ia_k|a_j x},\\
                H_2(s) &:=\set{x: \textrm{ there are } i<j<k \textrm{ such that no $D$-relations on $\set{a_i, a_j, a_k, x}$ hold}},\\
                H_3(s) &:= \set{x: \textrm{ there are } i<j<k<\ell \textrm{ such that } a_ix|a_ka_\ell \textrm{ and } a_ia_j|xa_\ell}.
            \end{split}
        \end{equation*}
    \end{itemize}

    For an $n$-element indiscernible sequence $s$, note that each of its component sequences must be an indiscernible sequence of singletons. With this in mind, define the discernible hull of $s$ as the union of the discernible hulls of its component sequences.
\end{definition}

Before proceeding, let us remark that when $s$ is unbounded, the set $H_3$ actually contains the sets $H_1$ and $H_2$. However we still list the sets separately as the elements of $H_1 \setminus H_3$ and $H_2 \setminus H_3$ behave substantially different to the elements in $H_3$. We shall also give intuition for the nature of the three sets $H_1, H_2, H_3$ which form the hull of $s$ when $s$ is a monotonic sequence indexed by, say, the integers. We may think of $s$ as the leaves of the following infinite tree, as in Theorem \ref{D-sets are trees}:

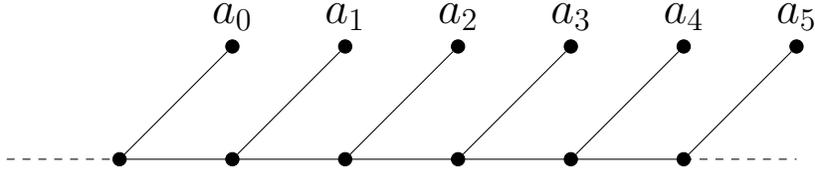
 \begin{center}
\begin{tikzpicture}[scale=1.5,vertex/.style={circle,draw,fill=black,inner sep=0pt,minimum size=5pt}]

    \foreach \i in {0, 1, 2, 3, 4, 5} {
        \node[vertex] (x\i) at (\i,0) {};
        \node[vertex, label={[font=\huge]above:$a_{\i}$}] (y\i) at (\i+1,1) {};
        \draw (x\i) -- (y\i);
    }
    \draw (0, 0) -- (5, 0);
    \draw[dashed] (5,0) -- (6, 0);
    \draw[dashed] (-1,0) -- (0, 0);

    \end{tikzpicture}
    
    \end{center}

Now, if we wish to attach an $x \in H(s)$ to this tree, it may happen in essentially three ways: either $x$ attaches between a leaf and inner node of the tree, at an inner node of the tree, or between two inner nodes of the tree:
 
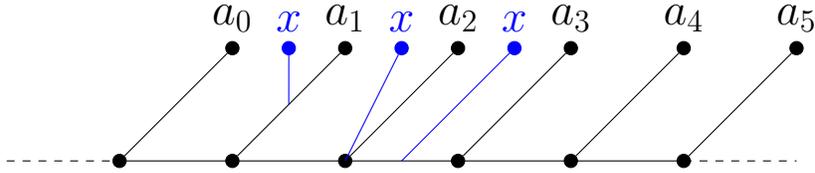
\begin{center}
\begin{tikzpicture}[scale=1.5,vertex/.style={circle,draw,fill=black,inner sep=0pt,minimum size=5pt}]
\foreach \i in {0, 1, 2, 3, 4, 5} {
        \node[vertex] (x\i) at (\i,0) {};
        \node[vertex, label={[font=\huge]above:$a_{\i}$}] (y\i) at (\i+1,1) {};
        \draw (x\i) -- (y\i);
    }
    \draw (0, 0) -- (5, 0);
    \draw[dashed] (5,0) -- (6, 0);
    \draw[dashed] (-1,0) -- (0, 0);

    \node[vertex, color=blue, label={[font=\huge,color=blue]above:$x$}] (x1) at (1.5, 1) {};
    \node[vertex, color=blue, label={[font=\huge,color=blue]above:$x$}] (x2) at (2.5, 1) {};
    \node[vertex, color=blue, label={[font=\huge,color=blue]above:$x$}] (x3) at (3.5,1) {};

    \draw[color=blue] (x1) -- (1.5, 0.5);
    \draw[color=blue] (x2) -- (2, 0);
    \draw[color=blue] (x3) -- (2.5, 0);
    \end{tikzpicture}
    
    \end{center}
When $x$ attaches itself between an inner node and a leaf node $a_j$, then it is in $H_1(s)$, witnessed by any $(a_i, a_j, a_k)$ where $i < j < k$, as in the leftmost $x$ in the drawing; when it attaches itself at an inner node $a_j$, then it is in $H_2(s)$, witnessed identically, as in the middle $x$ in the drawing; when it attaches itself between two inner nodes $a_j < a_k$, then it is (strictly) in $H_3(s)$ witnessed by any $(a_i, a_j, a_k, a_\ell)$ with $i < j < k < \ell$, as in the rightmost $x$ in the drawing. Thus, the only way to avoid being in $H(s)$ is to be ``a long way to the left or to the right'' such that it does not attach to the tree at all. These ideas are formalized in the following lemma:

\begin{lemma}\label{Characterizing the hull of monotonic sequences} Let $s$ be an unbounded monotonic sequence of singletons.
    \begin{enumerate}
    
        \item  Define $L(s)$ the \emph{left frontier} of $s$ (resp. $R(s)$ the \emph{right frontier} of $s$) as the intersection of all sectors $\Sigma$ such that $\Sigma \cap s$ is a nonempty initial (resp. final) segment of $s$. Then $\Omega \setminus H(s) = L(s) \cup R(s)$.

        \item If $x \in L(s)$, then $s' := s \cup \set{x}$, where $x$ is added as a minimum, is a monotonic sequence. If $x \in R(s)$ the same is true if added as a maximum. (This is the only case we will treat of an indiscernible sequence with an endpoint.)

        \item  If $x \in H_1(s) \cup H_2(s)$ with witnesses $a_i < a_j < a_k$, then the sequence $s' := (s \setminus \set{a_j}) \cup \set{x}$, where $x$ takes the same place in the order that $a_j$ held, is still a monotonic sequence.

        \item If $x \in H_3(s) \setminus (H_1(s) \cup H_2(s))$, then we may add $x$ to $s$ as a new element such that the sequence remains monotonic and unbounded. 
    \end{enumerate}

\end{lemma}
\begin{proof}\phantom{ }
    \begin{enumerate}
        \item  Suppose $x \in L(s)$, we shall prove that $x$ is not in $H_1(s), H_2(s)$ or $H_3(s)$ with one argument:
    
        Let $i < j < k$, let $\mathcal{C}$ be a splitting of $\Omega$ which extends $\set{\set{a_i}, \set{a_j}, \set{a_k}}$ and let $\Sigma$ be the sector of $\mathcal{C}$ containing $a_i$, then $\Sigma \cap s$ is the initial segment $(- \infty, a_j)$ and so $x \in \Sigma$, but then $a_ix|a_ja_k$ which means $\set{i,j,k}$ cannot witness that $x$ is in $H_1$ or $H_2$, and if we add an $\ell > k$, then necessarily $a_\ell$ belongs to the same sector of $\mathcal{C}$ as $a_k$ and we have $a_ix|a_ja_\ell$ which means $a_ia_j|xa_\ell$ is impossible and so $\set{i, j, k, \ell}$ cannot witness that $x \in H_3$. Hence $x$ is not in any of the sets and so $x \not \in H(x)$.

        By an analogous argument, no element of $R(s)$ is in $H(s)$. 
    
        Now we must prove no element of $H(s)$ is in $L(s)$ or $R(s)$. For this, suppose $x \in H(s)$. If $x \in H_1$ or $x \in H_2$, then let $i < j < k$ witness it. If $x \in H_3$, let $i < j < j' < k$ witness it. In any case extend $\set{\set{a_i}, \set{x}, \set{a_k}}$ to a node splitting $\mathcal{C}$ of $\Omega$. if $i' < j$ then we always have $a_{i'}a_i | a_j a_k$ and
        \begin{itemize}
            \item if $x \in H_1(s)$ then $x$ and $a_j$ share a sector in $\mathcal{C}$ and so $a_{i'} a_i | x a_k$;
            \item if $x \in H_2(s)$ then the splitting $\mathcal{C}$ extends the splitting $\set{\set{a_i}, \set{a_j}, \set{a_k}, x}$ and so again $a_{i'} a_i | x a_k$
            \item if $x \in H_3(s)$ then $a_j$ is in the same sector as $a_i$ (since $a_i a_j|x a_k$) and so $a_{i'} a_i | x a_k$.
        \end{itemize}
        And so in all cases we have that $a_{i'}$ is in the same sector $\Sigma$ as $a_i$ and so $\Sigma \cap s = (-\infty, a_j)$, an initial segment, and $\Sigma$ does not include $x$ by construction, so we conclude that $x \not \in L(s)$, and by an analogous argument $x \not \in R(s)$ and we are done.

        \item Suppose $x \in L(s)$. Let $i < j < k$ and let $\mathcal{C}$ be a splitting which separates $a_i, a_j$ and  $a_k$, then the sector containing $a_i$ intersects $s$ in the initial segment $(-\infty, a_j)$ and so contains $x$, which means $x a_i |a_j a_k$ and so the augmented sequence is indeed monotonic. The proof for $x \in R(s)$ is analogous.

        \item Suppose $x \in H_1(s) \cup H_2(s)$ with witnesses $a_i < a_j < a_k$, then $\set{\set{a_i}, \set{a_j}, \set{a_k}}$ extends to a splitting such that (if $H \in H_1(s))$) $x$ shares a sector with $a_j$ or (if $x \in H_2(s))$ $x$ does not share a sector with any of the witnesses, and notice that by monotonicity the sector containing $a_i$ contains everything in $s$ lesser than $a_j$ and the sector containing $a_k$ contains everything in $s$ greater than $a_k$, so that the monotonicity condition holds on any tuple involving $x$ and three elements of $s$.

        \item Suppose $x \in H_3(s) \setminus (H_1(s) \cup H_2(s))$, with $a_i < a_j < a_k < a_\ell$ witnessing the fact that $x \in H_3(s)$. Extend $\set{\set{a_j}, \set{a_k}, \set{x}}$ to a splitting $\mathcal{C}$. If $j < m < k$, since $x \not \in H_1(s)$ then some $D$-relation holds on $\set{a_j, a_k, a_m, x}$, and since $x \not \in H_2(s)$ that relation cannot be $a_ja_k|xa_m$, so it is either $a_ja_m|a_kx$ or $a_k a_m|a_j x$. The former implies that $a_m$ is in the same splitting of $\mathcal{C}$ as $a_j$ and the latter that it is in the same splitting as $a_k$. So then the splitting $C$ contains two sectors $\Sigma_L$ and $\Sigma_R$ which together contain all of $s$, so one ($\Sigma_L$) intersects $s$ in an initial segment and the other ($\Sigma_R)$ in a final segment, creating, in effect, a Dedekind cut in $s$. We note that s cannot contain a supremum of $\Sigma_L$ or infimum of $\Sigma_R$ as such an element would necessarily be the $a_j$ that witnesses that $x \in H_1(s) \cup H_2(s)$, so that $x$ can be added to $s$ between $\Sigma_L$ and $\Sigma_R$ while preserving monotonicity.

        The condition that there are two elements less than $x$ and two elements greater than $x$ follows from the construction, since for any 4-tuple of witnesses the first two always end up less than $x$ and the last two end up greater than $x$.
    \end{enumerate}
\end{proof}

\begin{lemma}\label{single element type lemma for indiscernibility over a set}
    Let $s$ be an indiscernible sequence indexed by $I$ and let $B = \set{b_1, b_2, b_3}$ be a set disjoint from $H(s)$, then $D(a_ib_1; b_2b_3) \Leftrightarrow D(a_jb_1; b_2b_3)$ for all $i, j \in I$.
\end{lemma}
\begin{proof}
    If the sequence $s$ is constant this is trivial. Let us consider the nontrivial cases. Suppose $a_i$ is an arbitrary element of $s$. We shall prove its quantifier-free type over $B$ is fixed.

    If $s$ is petaled, then let $\mathcal{C}$ be the splitting which separates its elements. Note that $a_i$ belongs to some sector of $\mathcal{C}$ and by hypothesis $B$ is disjoint from this sector. If $b_1, b_2, b_3$ do not all lie in one sector of $\mathcal{C}$, then the $D$-relation $D(a_i b_1; b_2 b_3)$ holds if and only if and only if $b_2$ and $b_3$ share a sector and $b_1$ does not lie on that sector. On the other hand, if all $b_j$'s lie in some sector $\Sigma$ of $\mathcal{C}$, then if $\mathcal{C}'$ is the splitting which extends $\set{\set{b_1}, \set{b_2}, \set{b_3}}$, clearly $\mathcal{C} \neq \mathcal{C}'$. Note that $b_1, b_2, b_3$ are contained in different sectors of $\mathcal{C}'$ but the same sector $\Sigma$ of $\mathcal{C}$, so by the One Sector Lemma (Lemma \ref{the one sector lemma}) $\Sigma$ contains all but one sectors of $\mathcal{C}'$ (including the ones containing $b_1, b_2$ and $b_3$), and dually the remaining sector $\Sigma$' of $\mathcal{C}'$ contains all but one sectors of $\mathcal{C}$ (and the exception is $\Sigma$, a sector containing no elements of $s$). So we conclude $\Sigma'$ contains the entire sequence $s$, which implies that $\qftp(a_i/B)$ is fixed as $i$ varies, which implies the result.

    Now if $s$ is monotonic, if $b_1, b_2$ and $b_3$ are split among the left and right frontiers of $s$, as defined in Lemma \ref{Characterizing the hull of monotonic sequences}(1), then it is straightforward that $D(a_ib_1; b_2b_3)$ holds if and only if $b_2, b_3$ are in the same frontier and $b_1$ is in the other one (i.e. $b_1 \in L(s)$, $b_2, b_3 \in R(s)$ or vice-versa). On the other hand if $b_1, b_2, b_3$ are all in the left frontier (or the right one), then again let $\mathcal{C}'$ be a splitting which separates $b_1, b_2$ and $b_3$. Note that for any splitting $\mathcal{C}$, if $\mathcal{C}$ has a sector that intersects $s$ on an initial segment $s_0$, then it contains $L(s)$ and thus contains $\set{b_1, b_2, b_3}$, so by an application of the One Sector Lemma identical to the previous case we conclude that every element of $s$ to the right of $s_0$ is in some particular sector $\Sigma$ of $\mathcal{C}$, and this sector $\Sigma$ does not depend on $s_0$, so that we once again conclude that all of $s$ is contained within one sector of $\mathcal{C}'$, which implies that $\qftp(a_i/B)$ is fixed, as before.
\end{proof}
\begin{remark}
    If $\Omega$ has quantifier elimination in some language consisting of $D$ and some unary predicates, the previous lemma implies all elements of $s$ have the same type over any set $B$ disjoint from $H(s)$. In particular this is the case in a dense, regular $D$-set.
\end{remark}

\begin{definition}
    Let $\Omega$ be a $D$-set, $B$ be an arbitrary subset of $\Omega$ and $s$ be a sequence (of singletons or tuples) indexed by $I$. We say $s$ is weakly indiscernible over $B$ if it is indiscernible for quantifier-free formulas with parameters from $B$. Equivalently, whenever $\phi(\overline{x})$ is a formula with parameters from $B$ consisting of a single $D$-relation, $\overline{a}_{\overline{i}}$ is a tuple from $s$ whose elements are indexed by a tuple $\overline{i}$ from $I$, and $\overline{i'}$ has the same quantifier-free order type as $\overline{i}$, then we have $\phi(\overline{a}_{\overline{i}}) \Leftrightarrow \phi(\overline{a}_{\overline{i'}})$, where $\overline{a}_{\overline{i'}}$ is $\overline{a}_{\overline{i}}$ replacing each index from $\overline{i}$ with the corresponding index from $\overline{i'}$. Call a $\phi$ satisfying this requirement order-invariant.
\end{definition}
    For instance, if $s = (a_i)_{i \in I}$ is a sequence of singletons then for $s$ to be weakly indiscernible over $B$ requires that the formula $\phi(x_1, x_2) = D(b_1b_2,x_1x_2)$ be order-invariant, where $b_1, b_2 \in B$. That is, whenever $i<j$, $i'<j'$ are indexes in $I$, then $D(b_1, b_2, a_i, a_j) \Leftrightarrow D(b_1, b_2, a_{i'}, a_{j'})$.

\begin{remark}
    If the theory of $\Omega$ admits quantifier elimination, then weak indiscernibility is equivalent to indiscernibility. Moreover, if $\Omega$ is augmented by unary predicates so that its theory has quantifier elimination, then $s$ is indiscernible over $B$ if and only if it is weakly indiscernible over $B$ and all elements of $s$ fulfill the same unary predicates.
\end{remark}

\begin{theorem}\label{characterizing indiscernibility in one element}
    Let $s = (a_i)_{i \in I}$ be an unbounded indiscernible sequence of singletons in a (colored) $D$-set, and let $B$ be a set, then  $s$ is weakly indiscernible over $B$ if and only if $B \cap H(s) = \emptyset$.
\end{theorem}
\begin{proof}

    Let $\phi(x)$ be a formula with parameters on $B$ consisting of a single $D$-relation. Notice that if $\phi$ uses no parameters from $B$ (i.e. is of the form $D(x_1x_2;x_3x_4)$) then it is order-invariant by the indiscernibility of $s$, and if $\phi$ is a sentence (i.e. is of the form $D(b_1b_2;b_3b_4)$) then it is trivially order-invariant. In addition, Lemma \ref{single element type lemma for indiscernibility over a set} implies that whenever $B \cap H(s) = \emptyset$ and $\phi$ has a single variable, (i.e. is of the form $D(x b_1; b_2 b_3)$ or some symmetric equivalent), then $\phi$ is order-invariant. Thus, we may assume $\phi$ has either two or three variables. Lastly, notice we may always assume the variables in $\phi$ are distinct since $D$-relations involving repeated elements are always determined by the axioms.

    \begin{itemize}
        \item Case 1: $s$ is a constant sequence

        Trivial.

        \item Case 2: $s$ is a petaled sequence.

        Let $\mathcal{C}$ be the splitting which separates the elements of $s$.

        Suppose that $B \cap H(s) = \emptyset$. That is to say, a given sector of $\mathcal{C}$ contains either a single element of $s$, some nonempty subset of $B$, or is disjoint from $s$ and $B$. If $\phi$ has three variables and a parameter $b$, then $\phi(a_i, a_j, a_k)$ is false for any choice of three distinct $i, j, k$, since all elements in $\set{a_i, a_j, a_k, b}$ are from different sectors of $\mathcal{C}$. If $\phi$ has two variables and two parameters $b_1, b_2$, then the only possible $D$-relation on $\set{a_i, a_j, b_1, b_2}$ for $i\neq j$ is $a_ia_j|b_1b_2$ which occurs if and only if $b_1, b_2$ share a sector of $\mathcal{C}$. Either way $\phi$ is order-invariant.

        On the other hand, if $B \cap H(s) \neq \emptyset$, let $b \in B \cap H(s)$, let $\Sigma$ be the sector of $\mathcal{C}$ in which $b$ is located and let $a_i$ be the element of $s$ in that sector. Pick $j, k$ different from $i$ such that there exists an $i' \neq i$ with the same quantifier-free order type over $\set{j,k}$ as $i$ (e.g. Pick $j, k$ both much bigger than $i$). Note now that $D(a_ib ; a_ja_k)$ holds (since $a_i$ and $b$ share a sector) but $D(a_{i'}b ; a_ja_k)$ does not (all elements are in different sectors). Thus this formula is not order-invariant.

        \item Case 3: $s$ is a monotonic sequence.

        Suppose that $B \cap H(s) = \emptyset$, that is to say all elements of $B$ are contained either in $L(s)$ or $R(s)$, as defined in Lemma \ref{Characterizing the hull of monotonic sequences}(1). By Lemma \ref{Characterizing the hull of monotonic sequences}(2) if the parameters of $\phi$ contain at most one element of $B \cap L(s)$ and at most one element of $B \cap L(s)$, the truth value of $\phi(\overline{x})$ is determined by simply adding (if present) these elements of $L(s)$ and $R(s)$ to $s$ as minimum and maximum elements of $s$, respectively. By a similar logic, if there are exactly two parameters in $B \cap L(s)$, say $b_1$ and $b_2$, then $D(a_i a_j; b_1 b_2)$ holds always (and everything it implies by (D1) and (D2)). We have an analogous result for exactly two elements of $B \cap R(s)$. This covers all possible cases for a $\phi$ of two or three variables, so $\phi$ is order-invariant.

        On the other hand, if $B \cap H(s) \neq \emptyset$, then let $b \in B \cap H(s)$. If $b \in H_1(s) \cup H_2(s)$ then by Lemma \ref{Characterizing the hull of monotonic sequences}(3) we may replace an element $a_i$ of $s$ with $b$, but then the truth value of $D(b, a_j, a_k, a_\ell)$ will change based on which, if any, of the $j, k$ or $\ell$ are equal to $i$. Explicitly, if we choose $i', j,k \in I$ such that $j < i' < i < k$, then $a_j a_{i'} | b a_k$ but if $b \in H_1(s)$ then $a_ja_k|a_ib$ and if $b \in H_2(s)$ then no $D$-relations hold on $\set{b, a_i, a_j, a_k}$. That is to say $i'$ cannot be replaced by $i$ in either case which means $D(x_1x_2;bx_3)$ is not order-invariant. On the other hand if $b \in H_3(s) \setminus (H_1(s) \cup H_2(s))$, then by Lemma \ref{Characterizing the hull of monotonic sequences}(4) we may add $b$ into the middle of $s$ as a new element with index $i_b$, and if $i < j < i_b < k < \ell$, we have that $a_i b | a_k a_\ell$ but not $a_i b | a_j a_\ell$ which again means $D(x_1b;x_2x_3)$ is not order-invariant. 
    \end{itemize}
    We conclude that, indeed, $s$ is weakly invariant over $B$ if and only if $H(s) \cap B = \emptyset$.
\end{proof}

\begin{corollary}
    If the theory of $\Omega$ (possibly after augmenting with unary predicates) has quantifier elimination, then an indiscernible sequence of singletons $s$ is indiscernible over a set $B$ if and only if $H(s) \cap B = \emptyset$. In particular this holds in a dense, regular $D$-set.
\end{corollary}

Now we begin to generalize to indiscernible sequences of $n$-tuples. Recall that the discernible hull of a general indiscernible sequence is the union of the hulls of its constituent sequences. We shall first deal with the particular case that the component sequences have pairwise-disjoint hulls, which is a simple corollary:

\begin{corollary}\label{Characterizing indiscernibility in multiple variables with disjoint hulls}
    Let $\Omega$ be a $D$-set. Let $s = (s^1, \dots, s^n)$ be an ($n$-element) unbounded indiscernible sequence such that $H(s^i) \cap H(s^j) = \emptyset$ for $i \neq j$, and let $B$ be a set, then $s$ is weakly indiscernible over $B$ if and only if $B \cap H(s) = \emptyset$.
\end{corollary}
\begin{proof}
    The necessity of the condition follows immediately from Theorem \ref{characterizing indiscernibility in one element}, since if $B \cap H(s) \neq \emptyset$, then $B \cap H(s^i) \neq \emptyset$ for some $i$ and failure of weak indiscernibility follows.

    For sufficiency, it suffices to notice that under the hypotheses, each sequence $s^i$ is weakly indiscernible over $B$ and over the union of all other $s^j$'s by Theorem \ref{characterizing indiscernibility in one element}, so for any $\phi$ with parameters from $B$ which consists of a single $D$-relation, we may check that $s$ is order-invariant by simply shifting the elements component-by-component. That is, changing the indices of elements in $s^1$, then in $s^2$, and so on. By indiscernibility of each $s^i$ over the others and over $B$ this will preserve truth value at each stage, and at the end we will have ensured $\phi$ is order-invariant.
\end{proof}
\begin{remark}
    As before, in the presence of quantifier elimination this result may be strengthened to $s$ being indiscernible over $B$.
\end{remark}

The case of a general $n$-element indiscernible sequence will require a little more work, since there are many meaningfully-different ways in which its component hulls can fail to be disjoint. We shall build tools that allow to bypass this problem.

\begin{lemma}\label{Hulls are disjoint or identical}
    Suppose $s = (s^1, s^2)$ is an unbounded indiscernible sequence in a (colored) D-set, then either $H(s^1) \cap H(s^2) = \emptyset$ or  $H(s^1) = H(s^2)$. In the latter case $s^1$ and $s^2$ are both constant, both petaled or both monotonic.
    
    Additionally. If $s_1$ is petaled and $\mathcal{C}$ is the splitting which separates its elements, exactly one of the following holds:
    \begin{itemize}
        \item $H(s^1) = H(s^2)$, $\mathcal{C}$ also separates the elements of $s^2$ and $a^1_i$ shares a sector of $\mathcal{C}$ with $a^2_j$ if and only if $i = j$; or
        \item $H(s^1) \cap H(s^2) = \emptyset$, $s^2$ is petaled and $\mathcal{C}$ separates its elements; or
        \item $H(s^1) \cap H(s^2) = \emptyset$ and $H(s^2)$ is entirely contained within a single sector of $\mathcal{C}$.
    \end{itemize}
\end{lemma}
\begin{proof}
    We shall proceed by cases based on the nature of the sequences.

    \begin{itemize}

        \item Case 1: One of the sequences is constant.

        Trivially $H(s^1) \cap H(s^2) = \emptyset$.

        \item Case 2: One of the sequences, say $s^1$, is petaled. The other is not constant.

        Let $\mathcal{C}$ be the splitting which separates the elements of $s^1$. Notice that the statement ``$x$ and $y$ are in the same sector of $\mathcal{C}$'' is expressible as a first-order sentence with parameters from $s^1$. Namely, if $a_i^1, a_j^1$ are arbitrary, distinct elements of $s^1$, then the above sentence is defined by $D(xy;a_i^1a_j^1)$. 
        
        (Note: This is not quite correct, as it assumes the sector that $x$ and $y$ share is not one of the sectors which $a_i^1$ or $a_j^1$ occupy. This can be solved by adding two new elements $a_{i'}^1$, $a_{j'}^1$ distinct from the previous and letting the sentence be $D(xy;a_i^1a_j^1) \vee D(xy;a_{i'}^1a_{j'}^1)$. For the next part of the proof, however, we will use the simpler version and simply assume without loss of generality that $x$ and $y$ do not belong to one of the two problematic sectors)

        Let us first suppose $s_2 \cap H(s^1) \neq \emptyset$. That is, there is an $a_k^2 \in s^2$ which lies in the same sector of $\mathcal{C}$ as some $a_\ell^1 \in s^1$. We claim that then $\ell = k$. Suppose otherwise and pick $i,j \in I$ both distinct from $k$ and $\ell$ such that there is an $\ell' \neq \ell$ with the same quantifier-free order type over $\set{i, j, k}$ as $\ell$ (e.g. pick $i,j$ much bigger than $k, \ell$). We then have $D(a_k^2 a_\ell^1;a_i^1 a_j^1)$ and by indiscernibility we may change $\ell$ to $\ell'$, but then we are claiming $D(a_k^2 a_{\ell'}^1;a_i^1 a_j^1)$ which states $a_k^2$ and $a_{\ell'}^1$ share a sector of $\mathcal{C}$ which is absurd since $a_{\ell}^1$ and $ a_{\ell'}^1$ clearly don't. We conclude that indeed $\ell = k$, so that we have $D(a_k^2 a_k^1;a_i^1 a_j^1)$ (for \textit{any} choice of $i, j \neq k$). Then by indiscernibility this sentence is in fact true for any choice of $k$, and so we conclude that $a_k^1$ and $a_k^2$ always share a sector, that is $\mathcal{C}$ separates the elements of $s^2$ as well as $s^1$ and so $s^2$ is in fact petaled with the same hull as $s^1$. That is to say we fulfill the first possibility in the last part of the statement.

        Now let us instead suppose that $s_2 \cap H(s^1) = \emptyset$. Take three arbitrary distinct elements from $s^2$ and let $\mathcal{C}'$ be a splitting which separates them. If $\mathcal{C} \neq \mathcal{C}'$ then by the One Sector Lemma all but one sectors of $\mathcal{C}'$ are contained within one sector $\Sigma$ of $\mathcal{C}$. We claim this sector cannot be part of $H(s^1)$, since if it were there would be an $a_i^1 \in s^1$ which shares a sector of $\mathcal{C}$ with infinitely many elements of $s^2$, but as before this cannot be since indiscernibility would imply we can change the $i$ for some other index which is impossible, so that all of $s_2$ is contained within one sector of $\mathcal{C}$ which is not in $H(s^1)$, and it quickly follows from the definition of the hull and the One Sector Lemma that all of $H(s^2)$ is contained in this sector, so that $H(s^1) \cap H(s^2) = \emptyset$ and we fulfill the third possibility in the last part of the statement.

        On the other hand if $\mathcal{C} = \mathcal{C}'$ then we may assume without loss of generality that $s^2$ is petaled, since if it were monotonic we may pick our three elements differently to ensure $\mathcal{C} \neq \mathcal{C}'$ and repeat the previous argument, and under this supposition we now have that $\mathcal{C}'$ separates the elements of $s^2$ and so $H(s^1)$ and $H(s^2)$ are formed from the unions of two disjoint sets of sectors from one splitting and so $H(s^1) \cap H(s^2) = \emptyset$. That is we fulfill the second possibility in the last part of the statement.

        \item Case 3: Both sequences are monotonic.

        Recall that all elements of the hull of a monotonic sequence $s$ are in $H_3(s)$, that is if $x \in H(s)$ then there are $i < j < k < \ell$ such that $a_ix|a_ka_\ell$ and $a_ia_j|xa_\ell$. Let us say in such a case that $x$ is between $a_i$ and $a_j$. Note that to say that there is some element $x$ between some pair $a^1_i, a^1_j$ and also between some other pair $a^2_k, a^2_\ell$ is a first order sentence, so if there is any intersection between the hulls, then this first-order sentence will hold for those four elements, and so it will hold for \textit{any} four elements with the same order type, which readily implies that the hulls of $s^1$ and $s^2$ are identical.

    \end{itemize}
\end{proof}

Now we can prove the main theorem of this section:

\begin{theorem}\label{Indiscernibility over a set full characterization}
    Let $\Omega$ be a $D$-set. Let $s = (s^1, \dots, s^n)$ be an ($n$-element) unbounded indiscernible sequence indexed by $I$ on a (colored) $D$-set, and let $B$ be a set, then $s$ is weakly indiscernible over $B$ if and only if $B \cap H(s) = \emptyset$.
\end{theorem}
\begin{proof}
    Necessity follows directly from Corollary \ref{Characterizing indiscernibility in multiple variables with disjoint hulls}.

    We shall prove sufficiency similarly to \ref{characterizing indiscernibility in one element}. We assume $B \cap H(s) = \emptyset$ and, as in that proof, we shall check the order-invariance of a formula $\phi$ with parameters from $B$ consisting of a single $D$-relation, and by the same arguments as before assume $\phi$ has two or three variables which are always distinct. 
    
    For convenience, we shall treat formulas which pull from different sequences as different. That is to say. We consider a $D(x_1 b_1;b_2b_3)$ where $x_1$ is replaced by an element of $s^1$ to be a distinct formula from a $D(x_1 b_1;b_2b_3)$ where $x_1$ is replaced by an element of $s^2$. To denote this, we enrich the variables with superscripts that denote which sequence they may take values in, so the former formula is denoted by $D(x^1_1 b_1; b_2b_3)$ and is to be seen as distinct from the latter which is denoted by $D(x^2_1b_1;b_2b_3)$. We will always assume that at least two of the variables of $\phi$ are from different component sequences, say $s^1$ and $s^2$, such that $H(s^1)$ and $H(s^2)$ are not disjoint (and hence in fact $H(s^1) = H(s^2)$ and they are of the same type by Lemma \ref{Hulls are disjoint or identical}). We may do this since if this doesn't happen, Corollary \ref{Characterizing indiscernibility in multiple variables with disjoint hulls} yields order-invariance immediately.

    We shall denote the shared hull of $s^1$ and $s^2$ by $H$ and we will bifurcate on whether the sequences $s^1$ and $s^2$ are constant, petaled or monotonic.

    \begin{itemize}
        \item Case 1: $s^1$ and $s^2$ are constant.

        Trivial.

        \item Case 2: $s^1$ and $s^2$ are petaled.

        Let $\mathcal{C}$ be the splitting which separates the elements of $s^1$ (or equivalently, of $s^2$). 

        First let us deal with the case where $\phi$ has two variables and two parameters $b, b' \in B$, that is $\phi$ is a $D$-relation on $\set{x_1^1, x_2^2, b, b'}$ where recall the superscripts denote which $s^i$ the variables belong to. Notice that $x_1^1$ and $x_2^2$ can never share a sector of $\mathcal{C}$ with $b$ or $b'$. So for a given choice of $x_1^1 = a^1_i, x_2^2 = a_j^2$, the only $D$-relation which can hold on this set is $a_i^1 a_j^2 | b b'$ which occurs if and only if $b$ and $b'$ share a sector of $\mathcal{C}$ or if $a_i^1$ and $a_j^2$ do (i.e. if $i = j$ by Lemma \ref{Hulls are disjoint or identical}). In either case, the formula can be seen to be order-invariant.

        Now suppose $\phi$ has three variables and a parameter $b$. By hypothesis one variable is in $s^1$ and another in $s^2$. The third is in some $s^m$ (possibly $m = 1$ or $2$). Notice that if $H(s^m) \neq H$, then we may repeat the previous argument replacing $b$ for a given $a_m^k \in s^m$. By Lemma \ref{Hulls are disjoint or identical} either all of $s^m$ is contained within a single sector of $\mathcal{C}$ so that $b'$ sharing a sector of $\mathcal{C}$ with $a_k^m$ is independent of $m$, or $s^m$ is a petaled sequence which is also separated by $\mathcal{C}$, so that $b$ and $a_k^m$ can never share a sector since such a sector would be a part of $H(s^m)$. Either way the same conclusion follows.
        
        So we may assume $H(s^m) = H$. That is, all three variables are in petaled sequences with the same hull. Notice that $b$ is not in the same sector of $\mathcal{C}$ as any of the variables, so by the Sector Indifference Principle (Lemma \ref{sector indifference principle}) we may replace it with some element of $s^1$ with the same property. That is, for a given choice of $x_1^1 = a^1_i, x_2^2 = a_j^2, x_3^m = a_k^m$, if $a_{\ell}^1$ is chosen to not share a sector with any of the previous elements (i.e. if $\ell$ is chosen different from $i, j, k$), then $\phi(a^1_i, a^2_j, x_3^m; b) \Leftrightarrow \phi(a^1_i, a^2_j, x_3^m; a_{\ell}^1)$. The latter formula is order-invariant by the indiscernibility of $s$ over $\emptyset$, so it follows that the former is too. This covers all cases for petaled $s^1$ and $s^2$.

        \item Case 3: $s^1$ and $s^2$ are monotonic.

        Fix values for the variables in $\phi$.

        Any parameter $b$ of $\phi$ and any variable in an $s^k$ whose hull is not $H$ is either in $L = L(s^1)$ or $R = R(s^1)$ as defined in Lemma \ref{Characterizing the hull of monotonic sequences}(1). Replace any such element in $L$ with $a_j^1$ where $j$ is some index in $I$ smaller than any index which has appeared already, and replace any in $R$ with $a_k^1$ where $k$ is some index \textit{larger} than any which has appeared already. These steps are to be performed in sequence, so that if two elements are in $L$ (or $R$), then their replacements have different indexes.

        After these replacements, all the $x_i$ are in $s$. It is easy to check that this new formula has the same truth value as the previous and that if it is order-invariant, then so is the original, order-invariance of the new formula follows immediately from the fact that $s$ is indiscernible over $\emptyset$, so we are done. 
    \end{itemize}
\end{proof}

\begin{remark}\label{remark on Indiscernibility over a set full characterization}
    As before, in the presence of quantifier elimination, we may strengthen weak indiscernibility to indiscernibility.
\end{remark}

\begin{corollary}
    Let $s_1, s_2$ be two indiscernible sequences. If they are mutually indiscernible, then $s_1 \cap H(s_2) = \emptyset$,  $s_2 \cap H(s_1) = \emptyset$. The converse holds in any $D$-set augmented with unary predicates which has quantifier elimination.
\end{corollary}

\begin{lemma}
    Let $s_1, s_2$ be two indiscernible sequences. If $H(s_1) \cap H(s_2) \neq \emptyset$ then $s_1 \cap H(s_2) \neq \emptyset$ or $s_2 \cap H(s_1) \neq \emptyset$.
\end{lemma}
\begin{proof}
    Without loss we may assume both sequences are 1-element non-constant sequences over the same index set. Take an $x \in H(s_1) \cap H(s_2)$. We define two splittings $\mathcal{C}_1, \mathcal{C}_2$ as follows:
    \begin{itemize}
        \item If $s_i$ is petaled, define $\mathcal{C}_i$ as the splitting which separates the elements of $s_i$.

        \item If $s_i$ is monotonic, let $a^i_{j'}, a^i_j, a^i_k, a^i_{k'}$ be witnesses that $x \in H_3(s_1)$. That is these are elements of $s^1$ with $a^i_{j'} x| a^i_k a^i_{k'}$ and $a^i_{j'} a^i_j| x a^i_{k'}$. Take $\mathcal{C}_i$ as the splitting which extends $\set{\set{x}, \set{a^i_j}, \set{a^i_k}}$.
    \end{itemize}

    We shall bifurcate into cases:
    \begin{itemize}
        \item Case 1: $\mathcal{C}_1 \neq \mathcal{C}_2$.

        Then by the One Sector Lemma, one sector of each splitting contains all but one sector of the other. Recall that the union of these special sectors is the whole of $\Omega$, so in particular one of them, say a $\Sigma \in \mathcal{C}_1$, contains $x$. Notice that since $x \in H(s_1)$, $\Sigma$ is necessarily contained in $H(s_1)$ by the definition of the hull (in either case, petaled or monotonic). Notice also that at least two sectors of $\mathcal{C}_2$ contain elements of $s_2$, so $\Sigma$ necessarily contains elements of $s_2$.

        \item Case 2: $\mathcal{C}_1 = \mathcal{C}_2 =: \mathcal{C}$ and $s_1, s_2$ are both petaled.

        Then $x$ shares a sector of $\mathcal{C}$ with an $a_i \in s_1$ and a $b_j \in s_2$, and so $a_i \in H(s_2)$ (and $b_j \in H(s_1)$ as well).

        \item Case 3: $\mathcal{C}_1 = \mathcal{C}_2 =: \mathcal{C}$ and $s_1, s_2$ are both monotonic. 
        
        For each $i$, $\mathcal{C}$ has one sector $\Sigma_i^\ell$ which intersects $s_i$ in an initial segment and one sector $\Sigma_i^r$ which intersects $s_i$ in a final segment. Anything outside of these two sectors is automatically in $H_3(s_i)$, as seen by taking as witnesses two pairs from $s_i$, one from $\Sigma_i^\ell$ and one from $\Sigma_i^r$. Thus, if $\set{\Sigma_1^\ell, \Sigma_1^r} \neq \set{\Sigma_2^\ell, \Sigma_2^r}$ we are done, so assume otherwise. Without loss of generality (flipping one order if necessary) we may further assume that $\Sigma_1^\ell = \Sigma_2^\ell =: \Sigma^\ell$ and $\Sigma_1^r = \Sigma_2^r =: \Sigma^r$.

        Now suppose $s_1 \cap H(s_2) = \emptyset$. Take two elements $a_j^1, a_{j'}^1$ in $\Sigma^\ell \cap s_1$. By assumption they are not in $H(s_2)$ and so by Lemma \ref{Characterizing the hull of monotonic sequences} they must be in $L(s_2)$, since they can't be in $R(s_2)$ given that $\Sigma^r$ intersects $s_2$ in a initial segment. Similarly, take two elements $a_k^1, a_{k'}^1$ in $\Sigma^r \cap s_1$. They must be in $R(s_2)$. But then for any $y \in s_2$ we have that $y \in H_3(s_1)$, witnessed by $a_j^1, a_{j'}^1, a_{k}^1, a_{k'}^1$ as a straightforward consequence of Lemma \ref{Characterizing the hull of monotonic sequences}, so that $s_2 \cap H(s_1)$ is all of $s_2$, and in particular is nonempty. We conclude that one of $s_1 \cap H(s_2)$ or $s_2 \cap H(s_1)$ is nonempty and we are done.

        \item Case 4: $\mathcal{C}_1 = \mathcal{C}_2 =: \mathcal{C}$, $s_1$ is petaled and $s_2$ is monotonic.

        Take an element $a_i \in s_1$ which does not share a sector with any elements of $s_2$ (which exists because only two sectors of $\mathcal{C}$ contain elements of $s_2$ whereas infinitely many of them contain elements of $s_1$). Notice that $a_i \in H_3(s_2)$, witnessed by the same elements which witness $x \in H_3(s_2)$.
    \end{itemize}
\end{proof}

\begin{corollary}
\label{disjoint_hulls}
    If two indiscernible sequences are mutually indiscernible, then their hulls are disjoint.
\end{corollary}

\begin{remark}
    In the presence of quantifier elimination, we have a partial converse. It is possible to have two sequences $s_1, s_2$ with disjoint hulls but $s_1\cap H(s_2) \neq \emptyset$. However, this is only possible if $s_1$ has a constant component which lies on $H(s_2)$, so the converse result is that in the presence of quantifier elimination, the requirements for mutual indiscernibility are that the hulls are disjoint and that the constant components of each sequence are disjoint from the hull of the other. (Note that the constant components of each sequence may intersect). 
\end{remark}

\begin{corollary}
    Let $T$ be the theory of a $D$-set augmented with unary predicates such that $T$ has quantifier elimination. Then $T$ is distal if and only if no model of $T$ contains an infinite petaled sequence.
\end{corollary}
\begin{proof}
    Let $s = s_1 + s_2$ be an unbounded indiscernible sequence of singletons over a set $B$, and let $c$ be such that $s_1 + c + s_2$ is indiscernible over $\emptyset$.

    If $s$ is constant then the conditions of distality are trivial.

    If $s$ is monotonic, notice that $c \in H(s)$ since for $a_1, a_1' \in s_1$ and $a_2, a_2' \in s_2$, we have $a_1a_1'|c a_2'$ and $a_1c |a_2 a_2'$, hence $H(s_1 + c + s_2) = H(s)$ and so $H(s_1 + c + s_2) \cap B = \emptyset$ and $s_1 + c + s_2$ is indiscernible over $B$.

    Thus, if $T$ admits only constant and monotonic indiscernible sequences, then it is distal. However, if $s$ is petaled, let $a$ be an element of $s$ such that $s$ can be written as $s_1 + c + s_2$ with $s_1, s_2$ unbounded. Such an $a$ can be assumed to exist without loss of generality, since a petaled sequence can always be re-indexed to any linear order of its cardinality. Notice that $s_1 + s_2$ is indiscernible over $\set{c}$ since $H(s_1 + s_2) \not\ni c$ and that $s_1 + c + s_2$ is indiscernible over the empty set, but obviously not over $\set{c}$. Thus a $D$-set with an infinite petaled sequence is not distal.
\end{proof}

Finally, we prove Theorem~\ref{dpmin}.

\begin{proof}
    Suppose $T$ is a theory of colored $D$-sets with quantifier elimination. Given two indiscernible sequences $s_1, s_2$ and some $b$ in a model $\Omega$ of $T$, we have that each $s_i$ is indiscernible over $b$ if and only if $b \not \in H(s_i)$. Since $s_1$ and $s_2$ are mutually indiscernible we also have $H(s_1) \cap H(s_2) = \emptyset$ and so at least one of $s_1, s_2$ is indiscernible over $b$.
\end{proof}

\bibliographystyle{babplain}
\bibliography{ref}

\end{document}